\documentclass[hidelinks]{siamart220329}

\title{Uniform $\mathcal{H}$-matrix Compression with Applications to Boundary Integral Equations\thanks{Submitted to the journal's Software, High-Performance Computing, and Computational Methods in Science and Engineering section May 29, 2024.
\funding{This work was supported by FWO-Flanders project G088622N.}}}

\author{Kobe Bruyninckx\thanks{Department of Computer Science, KU Leuven, Leuven, 3000
  (\email{kobe.bruyninckx@kuleuven.be}, \email{daan.huybrechs@kuleuven.be}, \email{karl.meerbergen@kuleuven.be}).}
\and Daan Huybrechs\footnotemark[2]
\and Karl Meerbergen\footnotemark[2]}

\headers{Uniform $\mathcal{H}$-matrix Compression}{K. Bruyninckx, D. Huybrechs and K. Meerbergen}

\ifpdf
\hypersetup{
  pdftitle={Uniform $\mathcal{H}$-matrix Compression with Applications to Boundary Integral Equations},
  pdfauthor={K. Bruyninckx, D. Huybrechs and K. Meerbergen}
}
\fi

\usepackage{amssymb} 
\usepackage{mathrsfs} 
\usepackage{amsfonts} 
\usepackage{accents} 
\usepackage{graphicx}
\usepackage[algo2e,ruled,vlined,linesnumbered]{algorithm2e}
    \let\oldnl\nl 
    \newcommand{\nonl}{\renewcommand{\nl}{\let\nl\oldnl}} 
\usepackage{booktabs} 
\usepackage{siunitx,xfp} 
\usepackage{csvsimple} 
\usepackage{chessfss}
\usepackage{pgfplots}
\usepackage{pgfplotstable}
\pgfplotsset{compat=1.18}
\usepgfplotslibrary{groupplots}
\usepgfplotslibrary{fillbetween}
\usepackage[caption=false]{subfig}
\usepackage{float}

\newcommand{\bigO}{\mathcal{O}} 
\newcommand{\R}{\mathbb{R}} 
\newcommand{\C}{\mathbb{C}} 
\newcommand{\norm}[1]{\left\|#1\right\|} 
\newcommand{\vu}{\mathbf{u}} 
\newcommand{\vv}{\mathbf{v}} 
\newcommand{\vw}{\mathbf{w}} 
\newcommand{\restr}[2]{{#1}_{|#2}} 
\newcommand{\Hm}{\mathcal{H}} 
\newcommand{\UHm}{\mathcal{UH}} 
\newcommand{\DHm}{\mathcal{DH}} 
\newcommand{\ct}{\tau} 
\newcommand{\cs}{\sigma} 
\newcommand{\ItJ}{{I\times J}} 
\newcommand{\ItI}{{I\times I}} 
\newcommand{\ts}{{\ct\times\cs}} 
\newcommand{\st}{{\cs\times\ct}} 
\newcommand{\tree}{\mathcal{T}} 
\newcommand{\leaves}{\mathcal{L}} 
\newcommand{\Lrc}{\mathcal{L}_I'} 
\newcommand{\Lcc}{\mathcal{L}_J'} 
\newcommand{\row}[1]{\mathrm{row}(#1)} 
\newcommand{\col}[1]{\mathrm{col}(#1)} 
\newcommand{\csp}{c_\mathrm{sp}} 
\newcommand{\kb}{k_b} 
\newcommand{\kmax}{k_{\max}} 
\newcommand{\lt}{\ell_\ct} 
\newcommand{\ls}{\ell_\cs} 
\newcommand{\lmax}{\ell_{\max}} 
\newcommand{\kt}{k_\ct^\mathrm{tot}} 
\newcommand{\ks}{k_\cs^\mathrm{tot}} 
\newcommand{\x}{\mathbf{x}} 
\newcommand{\y}{\mathbf{y}} 
\newcommand{\fspace}[1]{\mathscr{#1}} 
\renewcommand{\d}[1]{\ensuremath{\operatorname{d}\!{#1}}} 

\newcommand{\supp}[1]{\mathrm{supp}#1}
\newcommand{\svdU}{\mathbf{U}}
\newcommand{\svdV}{\mathbf{V}}
\newcommand{\svdS}{\mathbf{\Sigma}}
\DeclareMathOperator*{\argmin}{arg\,min}

\def\CC{{C\nolinebreak[4]\hspace{-.05em}\raisebox{.4ex}{\tiny\bf ++}}}

\newsiamremark{remark}{Remark}

\crefname{algocf}{Algorithm}{Algorithms}
\Crefname{algocf}{Algorithm}{Algorithms}

\begin{document}
\maketitle

\begin{abstract}
Boundary integral equations lead to dense system matrices when discretized, yet they are data-sparse. Using the $\Hm$-matrix format, this sparsity is exploited to achieve $\bigO(N\log N)$ complexity for storage and multiplication by a vector. This is achieved purely algebraically, based on low-rank approximations of subblocks, and hence the format is also applicable to a wider range of problems. The $\Hm^2$-matrix format improves the complexity to $\bigO(N)$ by introducing a recursive structure onto subblocks on multiple levels. However, in many cases this comes with a large proportionality constant, making the $\Hm^2$-matrix format advantageous mostly for large problems. In this paper we investigate the usefulness of a matrix format that lies in between these two: Uniform $\Hm$-matrices. An algebraic compression algorithm is introduced to transform a regular $\Hm$-matrix into a uniform $\Hm$-matrix, which maintains the asymptotic complexity. Using examples of the BEM formulation of the Helmholtz equation, we show that this scheme lowers the storage requirement and execution time of the matrix-vector product without significantly impacting the construction time.
\end{abstract}



\section{Introduction}\label{sec:1_intro}

The storage requirements for a dense $N\times N$-matrix scale as $\bigO(N^2)$, which may become prohibitively large in applications. In some cases, such as finite element methods applied to partial differential equations (PDEs), the system matrix is naturally sparse in that only $\bigO(N)$ matrix entries are non-zero. In other applications matrices arise that are not sparse in the usual sense but still allow a data-sparse approximation. Examples are discretizations of integral equations or the solution operator, i.e.\ the inverse of a finite element matrix, for elliptic PDEs. Data-sparse hierarchical matrices ($\Hm$-matrices) are applicable in both cases \cite{Borm2005b,Bebendorf2008,Borm2010a}.

The hierarchical matrix framework, first introduced in \cite{Hackbusch1999}, allows for storage of and operations with matrices in $\bigO(N\log N)$ complexity. This is achieved by splitting the matrix in appropriate submatrices of which the majority can be approximately represented by matrices of low rank. The approximation is constructed algebraically from a subset of entries of the matrix, using so-called Adaptive Cross Approximation (ACA) and similar techniques, which are applicable in various cases.

For integral operators of the form
\begin{equation*}
    \begin{aligned}
        (\mathcal{A}u)(\x) = \int_{\Omega}{g(\x,\y)u(\y)\d s_\y}, && \Omega\subseteq \R^d
    \end{aligned}
\end{equation*}
similar techniques exist based on analytical expansions of the kernel function $g$. Originally, the Fast Multipole Method (FMM) \cite{Rokhlin1985,GreengardRokhlin1987,Greengard1988} and the panel clustering method \cite{HackbuschNowak1989,Sauter2000} used explicit expansions of $g$ which are kernel-specific. This drawback is removed in kernel-independent variants \cite{BormGrasedyck2004,YingBirosZorin2004,FongDarve2009,LetourneauCeckaDarve2014} but the utilized general expansions can become less efficient.

Improvement on the $\Hm$-matrix format have been developed in the form of $\Hm^2$-matrices \cite{Hackbusch2002}, which employ a multilevel structure onto the low-rank factorizations of the subblocks. This hierarchical storage structure can improve the complexity to $\bigO(N)$. Yet, to enable construction in linear time, most $\Hm^2$-matrix frameworks fall back to analytic tools. To retain the generality of the algebraic approach, they utilize polynomial approximations \cite{Borm2005c} which often overestimate the rank of the blocks. To alleviate this problem, schemes have been developed to recompress $\Hm^2$-matrices \cite{Hackbusch2002,Borm2004,Borm2005a,Borm2009a}. This can even be performed during construction, such that the original uncompressed $\Hm^2$-matrix is never fully stored in memory. The technique has also been extended to directional $\Hm$-matrices ($\mathcal{DH}$-matrices) \cite{Borm2017a,Borm2020,Borm2023}.

In the spirit of these recompression techniques, we investigate the compression of $\Hm$-matrices into the uniform $\Hm$-matrix format, which lies halfway between the $\Hm$- and $\Hm^2$-matrix formats. Uniform $\Hm$-matrices retain the $\bigO(N\log N)$ complexity of regular $\Hm$-matrices, but with an improved proportionality constant. This means that a constant compression factor is achieved with a relatively small implementation effort, owing to the simplicity of uniform $\Hm$-matrices. This is in contrast to the more involved $\Hm^2$-matrix format which can achieve better results but has to rely on a combination of analytic construction and algebraic recompression to do so. While not asymptotically optimal, uniform $\Hm$-matrices retain the appeal of regular $\Hm$-matrices in that they are much easier to use. As we shall show, the format also enables a simple parallellization of the matrix-vector product. Interestingly, our numerical results indicate that the compression factor of memory usage leads to a similar improvement of the timings of the matrix-vector product.

Since its original introduction together with the regular $\Hm$-matrix \cite{Hackbusch1999}, the uniform $\Hm$-matrix format has been explored very little. In \cite{Bebendorf2009}, a recompression technique based on Chebyshev polynomials is proposed that leads to uniform $\Hm$-matrices, but reduction in memory usage is only achieved when doing considerable reassembly during the matrix-vector product. Apart from this reference, uniform $\Hm$-matrices are introduced as an intermediate step towards $\Hm^2$-matrices \cite{Hackbusch2002,Borm2016}, but do not seem to have been an object of study in their own right. With the above-mentioned advantages in mind, in this paper we aim to illustrate that any application using $\Hm$-matrices may, with little effort, benefit from using uniform $\Hm$-matrices instead.

During the writing of this paper, it has come to our attention that similar (unpublished) work on uniform $\Hm$-matrices has been independently done by R. Kriemann. An implementation is available in the codebase \href{http://libhlr.org}{libHLR}\footnote{Specific results on uniform $\Hm$-matrices can be found at \href{http://libhlr.org/programs/uniform}{libhlr.org/programs/uniform}.} and preliminary results indicate, as we describe in more detail in this paper, that the uniform $\Hm$-matrix format results in memory reduction. Recent work by Kriemann on memory reduction includes mixed-precision compression of structured matrix formats in \cite{KriemannPreprint2023}.

The structure of the paper is as follows. \Cref{sec:2_hmatrices} introduces regular and uniform $\Hm$-matrices in the general setting. In \cref{sec:3_bem}, the Boundary Element Method and the $\Hm$-matrices that arise from this application are discussed. Then, \cref{sec:4_alg} proposes a technique to compress regular $\Hm$-matrices into uniform $\Hm$-matrices. Error analysis on this compression is provided in \cref{sec:5_analysis}. Lastly, \cref{sec:6_exp} presents numerical results on the compression scheme applied to matrices that originate from a boundary integral formulation of the Helmholtz equation and \cref{sec:7_conclusion} gives a brief conclusion.

\section{Hierarchical matrices}\label{sec:2_hmatrices}

This section gives a brief introduction to regular (\cref{ssec:2.1_reg}) and uniform (\cref{ssec:2.2_unf}) $\Hm$-matrices. For a more in-depth exposition on the subject, we refer to the books by Bebendorf \cite{Bebendorf2008book} and Hackbusch \cite{Hackbusch2015book}.

\subsection{Regular $\Hm$-matrices}\label{ssec:2.1_reg}
The $\Hm$-matrix format is based on partioning a matrix $A\in\C^\ItJ$ into submatrices such that most can be approximately represented by matrices of low rank. As $\Hm$-matrix methods aim to approximate the matrix in quasilinear complexity, i.e. quasilinear in the dimensions $|I|$ and $|J|$ of the matrix, subdividing the set $\ItJ$ directly is infeasible. Typically, the partition is induced by hierarchical subdivisions of the row and column indices $I$ and $J$ separately.

A hierarchy of subsets of an index set $I$ is defined using a \emph{cluster tree} $\tree_I$. 
\begin{definition}[cluster tree]\label{def:cluster_tree}
    A cluster tree $\tree_I$ for an index set $I$ is a tree with $I$ as its root, that employs a subdivision strategy $s:\mathcal{P}(I) \to \mathcal{P}(I)\times\dots\times\mathcal{P}(I)$\footnote{$\mathcal{P}(I)$ denotes the power set of $I$.} at each non-leaf node such that
    \begin{itemize}
        \item the children of a node form a partition of their parent,
        $$\ct=\bigcup_{\ct'\in s(\ct)}\ct' \quad\text{for all}~\ct\in\tree_I~\text{with}~s(\ct)\ne\emptyset;~\text{and}$$
        \item children of the same parent are disjoint,
        $$\ct_1\ne\ct_2 \Rightarrow \ct_1\cap\ct_2=\emptyset \quad\text{for all}~\ct\in\tree_I,~\ct_1,\ct_2\in s(\ct).$$
    \end{itemize}
    Nodes of a cluster tree are called \emph{clusters}. Its leaves are denoted by $\leaves(\tree_I)$.
\end{definition}
To each cluster $\ct\in\tree_I$, we can ascribe its level as the distance (in number of applications of $s$) to the root $I$ by $\mathrm{level}(\ct)$. Each level
$$\tree_I^{(\ell)}:=\{\ct\in\tree_I : \mathrm{level}(\ct)=\ell\}$$
of $\tree_I$ is a collection of disjoint subsets of $I$. The total number of levels in a cluster tree is called its depth and is denoted by $L(\tree_I):=\max_{\ct\in\tree_I} \mathrm{level}(\ct)+1$.

A pair of clusters $(\ct,\cs)$ defines a subset $\ts$ of $\ItJ$ and thus also a submatrix of $A\in\C^\ItJ$. From two cluster trees $\tree_I$ and $\tree_J$, a \emph{block cluster tree} can be constructed.
\begin{definition}[block cluster tree]\label{def:block_cluster_tree}
    Given cluster trees $\tree_I$ and $\tree_J$ for index sets $I$ and $J$, a block cluster tree $\tree_\ItJ$ for $\ItJ$ is a cluster tree of $\ItJ$ where
    \begin{itemize}
        \item each node in $\tree_\ItJ$ originates from a pair of cluster nodes in $\tree_I$ and $\tree_J$,
        $$\forall b\in\tree_\ItJ,~\exists \ct\in\tree_I, \cs\in\tree_J ~\text{s.t.}~b=\ts;~\text{and}$$
        \item for each non-leaf node $b=\ts\in\tree_\ItJ$, either $s(b)=s(\ct)\times s(\cs)$, $s(b)=\{\ct\}\times s(\cs)$ or $s(b)=s(\ct)\times\{\cs\}$ holds.
    \end{itemize}
    Nodes of a block cluster tree are called \emph{blocks}. Given $b=\ts\in\tree_\ItJ$, $\ct\in\tree_I$ and $\cs\in\tree_J$ are referred to as its row and column cluster respectively.
\end{definition}

A block cluster tree $\tree_\ItJ$ induces a partition of $\ItJ$ through its leaves
$$\mathcal{L}(\tree_\ItJ):=\{b\in\tree_\ItJ : s(b)=\emptyset\}.$$
This partition corresponds to a subdivision of matrix $A\in\C^\ItJ$ into submatrices which we denote by $\restr{A}{\ts}\in\C^\ts$ for all $\ts\in\mathcal{L}(\tree_\ItJ)$.

Constructing a $\Hm$-matrix approximation of $A$ is achieved by finding an approximate low-rank factorization for each submatrix $\restr{A}{\ts}$. However, in general, not all these submatrices will have approximate low rank. Therefore, the leaves are divided into an \emph{admissible} part $\mathcal{L}^+(\tree_\ItJ)\subset\leaves(\tree_\ItJ)$ and \emph{inadmissible} (or dense) part $\mathcal{L}^-(\tree_\ItJ):=\leaves(\tree_\ItJ)\setminus\leaves^+(\tree_\ItJ)$.

The partitions $P_\ItJ:=\leaves(\tree_\ItJ)$, $P^+_\ItJ:=\leaves^+(\tree_\ItJ)$ and $P^-_\ItJ:=\leaves^-(\tree_\ItJ)$ allow us to state the definition of a \emph{hierarchical} matrix.
\begin{definition}[hierarchical matrix]\label{def:hmat}
A hierarchical matrix ($\Hm$-matrix) is a matrix $A\in\C^\ItJ$ together with a partition $P_\ItJ$ of $\ItJ$ with admissible subset $P^+_\ItJ\subset P_\ItJ$, such that for all blocks $\ts\in P^+_\ItJ$,
\begin{equation*}
    \exists X\in\C^{\ct\times k},Y\in\C^{\cs\times k}: \quad \restr{A}{\ts}=XY^*
\end{equation*}
where $k:=k(\ts)$ is variable across submatrices.
\end{definition}

Given a matrix $A\in\C^\ItJ$, we denote its $\Hm$-matrix approximation by $A^{\Hm}$.
In an implementation, to represent a $\Hm$-matrix one needs to store the low-rank submatrices as well as the dense ones,
$$\{(X_b,Y_b)\}_{b\in P^+_\ItJ}~\text{and}~\{\restr{A}{b}\}_{b\in P^-_\ItJ},$$
where we have added subscripts to $X$ and $Y$ to distinguish between factorizations. The amount of storage this requires is characterized by the defining partitions and the rank distribution ${(k_b)}_{b\in P^+_\ItJ}$ of the factorizations. 

\subsection{Uniform $\Hm$-matrices}\label{ssec:2.2_unf}
In regular $\Hm$-matrices, as defined above, the low-rank factorizations of the admissible blocks are independent of each other. This setup is flexible but does not allow sharing of information across blocks. Storage requirements can be reduced by introducing a basis associated with each cluster, and to store a low-rank submatrix in terms of the bases of its row and column clusters.

The only clusters to consider are those in $\tree_I$ or $\tree_J$ that actually occur in admissible leaves of $\tree_\ItJ$. They are respectively
\begin{equation}\label{eq:clusters_in_partition}
\begin{aligned}
    \Lrc:=\Lrc(\tree_\ItJ)=\{\ct\in\tree_I : \exists\cs\in\tree_J~~\text{s.t.\ }~\ts\in\leaves^+(\tree_\ItJ) \}, \\
    \Lcc:=\Lcc(\tree_\ItJ)=\{\cs\in\tree_J : \exists\ct\in\tree_I~~\text{s.t.\ }~\ts\in\leaves^+(\tree_\ItJ) \}. 
\end{aligned}
\end{equation}
The \emph{cluster bases} are two families of matrices for these sets of clusters:
$$ \mathcal{U}:=\{U_\ct\in\C^{\ct\times \lt}\}_{\ct\in\Lrc}~\text{and}~\mathcal{V}:=\{V_\cs\in\C^{\cs\times \ls}\}_{\cs\in\Lcc},$$
in which $(\lt)_{\ct\in\Lrc}$ and $(\ls)_{\cs\in\Lcc}$ are their corresponding rank distributions. First established in \cite{Hackbusch1999}, the concept of row and column cluster bases introduces a new type of hierarchical matrix.

\begin{definition}[uniform hierarchical matrix]\label{def:uhmat}
A uniform hierarchical matrix ($\UHm$-matrix) is a matrix $A\in\C^\ItJ$ together with a partition $P_\ItJ$ of $\ItJ$ with admissible subset $P^+_\ItJ\subset P_\ItJ$, and with row and column cluster bases $\mathcal{U}$ and $\mathcal{V}$, such that for all blocks $b=\ts\in P^+_\ItJ$:
\begin{equation*}
    \exists S_b\in\C^{\lt\times \ls} : \quad \restr{A}{b}=U_\ct S_b V_\cs^*.
\end{equation*}
\end{definition}

Given a matrix $A\in\C^\ItJ$, its $\UHm$-matrix approximation is denoted by $A^{\UHm}$. Similarly to the $\Hm$-matrix format, a $\UHm$-matrix can be represented by storing a collection of matrices. In this case, these are
$$\mathcal{U}=\{U_\ct\}_{\ct\in\Lrc},\mathcal{V}=\{V_\cs\}_{\cs\in\Lcc},\{S_b\}_{b\in P^+_\ItJ}~\text{and}~\{\restr{A}{b}\}_{b\in P^-_\ItJ}.$$
The storage of the inadmissible blocks is unchanged compared to $\Hm$-matrices.

We introduce a measure to assess the possible reduction in storage when going from an $\Hm$-matrix to a $\UHm$-matrix, namely the maximum number of admissible blocks associated with a given cluster.
\begin{definition}[sparsity constant]\label{def:sparse_const}
Let $\tree_I$ and $\tree_J$ be cluster trees for the index sets $I$ and $J$ and let $\tree_{\ItJ}$ be a block cluster tree for $\ItJ$. Using $\Lrc:=\Lrc(\tree_\ItJ)$ and $\Lcc:=\Lcc(\tree_\ItJ)$ as defined in \cref{eq:clusters_in_partition}, we introduce the sets
\begin{equation*}
\begin{aligned}
        \row{\ct} := \{ \cs\in\tree_J ~:~ \ts\in\leaves^+(\tree_\ItJ) \}, &&& \ct\in\Lrc, \\     
        \col{\cs} := \{ \ct\in\tree_I ~:~ \ts\in\leaves^+(\tree_\ItJ) \}, &&& \cs\in\Lcc.
\end{aligned}
\end{equation*}
The \emph{sparsity constant} $\csp$ of a block cluster tree $\tree_{\ItJ}$ is 
$$ \csp(\tree_{\ItJ}):=\max\left\{\max_{\ct\in\Lrc}|\row{\ct}| ,\max_{\cs\in\Lcc}|\col{\cs}|\right\},$$
i.e.\ the maximum number of admissible leaf blocks $\ts\in \leaves^+(\tree_{\ItJ})$ associated with any row cluster $\ct\in\Lrc\subset\tree_I$ or column cluster $\cs\in\Lcc\subset\tree_J$.
\end{definition}

The sparsity constant allows to bound the storage requirements of both $\Hm$- and $\UHm$-matrices.
\begin{lemma}[Storage cost of hierarchical matrices]\label{lem:storage_cost}
Assume a matrix $A$ expressible in both $\Hm$-matrix and $\UHm$-matrix format with admissible partition $P^+_\ItJ$, and with maximum block rank $k_{\max}$ and maximum cluster rank $\ell_{\max}$ respectively.
The storage cost, or equivalently the total number of matrix elements to be stored in a representation, for the admissible blocks in the two formats, denoted by $N_\mathrm{st}^\Hm(A)$ and $N_\mathrm{st}^\UHm(A)$ respectively, are bounded by
\begin{subequations}
\begin{align}
    N_\mathrm{st}^\Hm(A) &\le \csp \kmax \left( L(\tree_I)|I| + L(\tree_J)|J| \right) \label{eq:storage_cost_H} \\
    N_\mathrm{st}^\UHm(A) &\le \lmax \left( L(\tree_I)|I| + L(\tree_J)|J| \right) + \lmax^2 2\csp \min\{|I|,|J|\}/n_\mathrm{min} \label{eq:storage_cost_U}
\end{align}
\end{subequations}
where $n_{\min}$ is the minimal size of a cluster in both $\tree_I$ and $\tree_J$.
\end{lemma}
\begin{proof}
Expression \cref{eq:storage_cost_H} corresponds to the storage of matrices $\{(X_b,Y_b)\}_{b\in P^+_\ItJ}$ of which the ranks are bounded by $k_\mathrm{max}$,
\begin{equation}\label{eq:storage_proof_H}
	N_\mathrm{st}^\Hm(A) \le \kmax\sum_{b\in P^+_\ItJ}\left(|\ct|+|\cs|\right) \le \csp \kmax \left( \sum_{\ct\in\Lrc}|\ct| + \sum_{\cs\in\Lcc}|\cs| \right).
\end{equation}
We use that each row and column cluster occurs at most $\csp$ times in the admissible partition. The two sums in the right hand side are bounded by $L(\tree_I)|I|$ and $L(\tree_J)|J|$ respectively, by first noticing that $\Lrc\subseteq \tree_I$ (and $\Lcc\subseteq \tree_J$) and then recalling that each level of a cluster tree is a set of disjoint subsets of its root. Summing over the levels proves the bound. 

For \cref{eq:storage_cost_U}, the storage contributions originate from cluster bases $\mathcal{U}$, $\mathcal{V}$ and coefficient matrices $\{S_b\}_{b\in P^+_\ItJ}$. The storage of $\mathcal{U}$ and $\mathcal{V}$ results in the same sums as in \cref{eq:storage_proof_H} but with $\lmax$ as constant in front. The coefficient matrices have an associated storage cost bounded by $\sum_{b\in P^+_\ItJ}\lmax^2$. In \cite[\S1.5]{Bebendorf2008book} it is shown that the number of blocks in a partition $P^+_\ItJ=\leaves^+(\tree_{\ItJ})$ is bounded by $2\csp \min\{|I|,|J|\}/n_\mathrm{min}$.
\end{proof}
\begin{remark}
	\Cref{lem:storage_cost} omits the storage cost of the dense blocks as they contribute an equal amount to both formats. Namely, it can be shown that this cost is bounded by $8\csp^{\mathrm{dns}} n_{\mathrm{min}}\min\{|I|,|J|\}$ if $|\ct|\le 2n_{\mathrm{min}}$ is assumed for all $\ct\in\leaves(\tree_I)\cup\leaves(\tree_J)$. Sparsity constant $\csp^{\mathrm{dns}}$ indicates the maximum number of dense blocks a cluster occurs in. The proofs by Bebendorf \cite[Theorem 2.6]{Bebendorf2008book} and Hackbush \cite[Lemma 6.13]{Hackbusch2015book} -- on which the above proof is also based -- incorporate the dense contribution into the log-linear part to simplify the bound on the total memory usage of the $\Hm$-matrix.
\end{remark}

While \cref{lem:storage_cost} only provides bounds, it indicates that possible storage reduction in using the $\UHm$-matrix format can be attributed to the sparsity constant $\csp$. Suppose an $\Hm$-matrix has constant ranks and constant sparsity constants. Disregarding the smaller contribution of the coefficient matrices, a factor of up to $\csp$ in compression can be achieved if all submatrices with the same row or column cluster share a basis. For general $\Hm$-matrices the latter might not be the case, as the admissible blocks may be arbitrary. For matrices arising from the Boundary Element Method, however, analysis hints at these submatrices possibly sharing a basis.

\subsection{Adaptive Cross Approximation}\label{ssec:2.3_ACA}
It is not specified in \cref{ssec:2.1_reg} how the low-rank factorizations of the regular $\Hm$-matrix are obtained. A popular choice to produce the block-wise factorizations $\restr{A}{b}\approx X_bY_b$ is Adaptive Cross Approximation (ACA). It constructs a low-rank approximation to $\restr{A}{b}$ by sampling rows and columns associated with pivots. It is an entirely numerical procedure which makes no (analytical) assumptions about where the matrix is coming from. Depending on the pivot selection strategy, different variants of ACA have been described. These include but are not limited to ACA with \emph{partial pivoting} \cite{Bebendorf2003} and ACA+ \cite{Grasedyck2005,Borm2005b}. 

For a rank $k$ approximation to a block $b=\ts$, ACA requires at most
$$C_{\mathrm{aca}}^{(1)}k \big(|\ct|+|\cs| \big) \left( C_\mathrm{elem} + C_{\mathrm{aca}}^{(2)} k \right)$$
operations, with $C_\mathrm{elem}$ the operation count of sampling one element of the matrix \cite[\S3.2]{Dirckx2024}. In the Galerkin BEM setting (cf.\ infra), $C_\mathrm{elem}=\bigO(q^4)$ when using quadrature of order $q$. The constants $C_{\mathrm{aca}}^{(1)}$ and $C_{\mathrm{aca}}^{(2)}$ depend on the specific version of ACA used.

The ACA algorithms typically do not guarantee that the obtained factorization is rank-optimal in the sense that its approximate rank w.r.t.\ the specified tolerance might be lower than its dimension $k$ seems to imply. Using the QR-decompositions of both $X_b$ and $Y_b$, the truncated SVD of the factorization can be cheaply computed in 
$$C_{\mathrm{arc}}^{(1)} k^2 \big(|\ct|+|\cs|\big) + C_{\mathrm{arc}}^{(2)} k^3$$
operations\footnote{The constants $C_{\mathrm{arc}}^{(1)}$ and $C_{\mathrm{arc}}^{(2)}$ depend on the algorithms used for the QRs and the SVD.} \cite{Grasedyck2003}. In this case, the representation of admissible blocks may take the form $\restr{A}{b}\approx U_b \Sigma_b V_b^*$ with $U_b$ and $V_b$ orthogonal and $\Sigma_b$ diagonal.

\section{Boundary Element Method}\label{sec:3_bem}

\subsection{Boundary integral equations and BEM}
The main application area for $\Hm$-matrices and related techniques is the efficient numerical solution of elliptic boundary value problems. In this setting, the differential equations of the boundary value problem are reformulated as boundary integral equations (BIEs) \cite[\S7]{Mclean2000} with integral operators $\mathcal{A}:\fspace{V}\to \fspace{W}$ of the form
$$ (\mathcal{A}u)(\x) = \int_{\Gamma}{g(\x,\y)u(\y)\d s_\y}$$
where $\Gamma:=\partial\Omega$ is the boundary of a (bounded Lipschitz) domain $\Omega\subset\mathbb{R}^d$ and $g(\cdot,\cdot)$ the \emph{kernel function} of the operator, typically a Green's function or a derivative thereof.

To solve the integral equation, it is discretized. The Galerkin formulation of the Boundary Element Method (BEM), based on a variational formulation, restricts the infinite dimensional function spaces $\fspace{V}$ and $\fspace{W}^*$ to finite dimensional subspaces $\fspace{V}_N\subset \fspace{V}$ with a basis $\{\phi_j\}_{j\in J}$ ($|J|=N$) and $\fspace{W}^*_M\subset \fspace{W}^*$ with a basis $\{\psi_i\}_{i\in I}$ ($|I|=M$) respectively. Characteristic in BEM is that each basis function has local support on $\Gamma$ and overlaps only with few other basis functions.

The Galerkin BEM discretization results in a system matrix $A$ with elements
\begin{equation*}
    a_{ij} = \langle A\phi_j,\psi_i\rangle = \int_\Gamma{\int_\Gamma{g(\x,\y) \psi_i(\x) \phi_j(\y) \d s_\x} \d s_\y},\quad i\in I, j\in J.
\end{equation*}
Two other ways of discretizing the integral equation are collocation methods and Nystr\o m methods. All three discretizations yield a dense matrix amenable to $\Hm$-matrix compression.

\subsection{$\Hm$-matrices for BEM}
Matrix $A$ is typically fully populated due to the non-local kernel function $g$. However, kernel functions originating from elliptic boundary value problems are \emph{asymptotically smooth} \cite[\S3.2]{Bebendorf2008book}, making them separable (or degenerate) when $\x$ and $\y$ are well-separated.  In that case $g$ can be approximated as
\begin{equation}\label{eq:sep_exp}
    g(\x,\y) \approx \sum_{\nu=1}^{r}{g_{\nu}^{(1)}(\x)g_{\nu}^{(2)}(\y)}.
\end{equation}

In the boundary element method, indices $i\in I$ and $j\in J$ are associated with parts of the computational domain $\Gamma$. For Galerkin, these are the supports of $\psi_i$ and $\phi_j$, denoted by $\supp_{\mathrm{row}}(i)=\supp{(\psi_i)}$ and $\supp_{\mathrm{col}}(j)=\supp{(\phi_j)}$. The supports of row and column clusters $\ct$ and $\cs$ are naturally
$$\supp{(\ct)}=\bigcup_{i\in\ct}\supp_{\mathrm{row}}(i) \quad\text{and}\quad \supp{(\cs)}=\bigcup_{j\in\cs}\supp_{\mathrm{col}}(j)$$
and the support of a cluster is regularly identified with itself.

If the above expansion \cref{eq:sep_exp} holds for all $\x\in\supp{(\ct)}$ and $\y\in\supp{(\cs)}$ of block $\ts$ up to some tolerance, the corresponding submatrix of $A$ will be of approximate rank (at most) $r$. Such blocks are identified using an \emph{admissibility criterion}. For kernel functions that are asymptotically smooth w.r.t.\ either $\x$ or $\y$, the condition
\begin{equation}\label{eq:strong_adm}
    \eta \, \text{dist}(\ct,\cs) > \max\{\text{diam}(\ct),\text{diam}(\cs)\}
\end{equation}
guarantees approximate separability \cite[\S3.3]{Bebendorf2008book}. Parameter $\eta$ specifies what is meant by ``sufficiently separated". An often-used relaxation to condition \cref{eq:strong_adm} if the kernel is asymptotically smooth w.r.t.\ both variables, is
\begin{equation}\label{eq:weak_adm}
    \eta \, \text{dist}(\ct,\cs) > \min\{\text{diam}(\ct),\text{diam}(\cs)\}.
\end{equation}

The admissibility criterion is used in BEM to identify leaves of the block cluster tree. If a block satisfies the admissibility criterion, it is labeled a leaf and becomes part of the admissible partition $P^+_\ItJ$. If not, the block is further subdivided unless it is too small. Then, it becomes part of the inadmissible leaves ($P^-_\ItJ$).

A log-linear representation of BEM matrices using the $\Hm$-matrix format -- log-linear in the total number of degrees of freedom, i.e., in the sizes $|I|$  and $|J|$ of the index sets -- relies on an appropriate partition $P_\ItJ$ and thus on well-chosen cluster trees $\tree_I$ and $\tree_J$. It is obvious from the admissibility criterion that the smaller the diameters of the clusters, the more likely it is for blocks to be admissible. Therefore the cluster trees are generally based on the geometry of the clusters such that subdivision minimizes the diameters of the children.

If an appropriate clustering strategy is employed, e.g.\ one based on principal component analysis (PCA) \cite{Pearson1901}, and the basis functions $\{\psi_i\}_{i\in I}$ and $\{\phi_j\}_{i\in J}$ have \emph{bounded overlap}, meaning that any point in their support is shared by at most a constant number of other basis functions, it holds that
\begin{itemize}
    \item the depths of the trees scale logarithmically with the index set sizes, i.e.\ $L(\tree_I) \sim \log|I|$, $L(\tree_J) \sim \log|J|$ and $L(\tree_\ItJ) \sim \min\{\log|I|,\log|J|\}$; and
    \item the sparsity constant $\csp$ is bounded independently of the sizes $|I|$ and $|J|$.
\end{itemize}
Proofs can be found in \S1.4 and \S1.5 of \cite{Bebendorf2008book} respectively\footnote{These proofs assume a level-conserving block cluster tree, implying restriction of the subdivision to $s(b)=s(\ct)\times s(\cs)$. However, the proofs can easily be extended to address block cluster trees with any of the three subdivision strategies in \cref{def:block_cluster_tree}.}, including precise conditions on the cluster trees. 

Together with \cref{lem:storage_cost}, the regular and uniform $\Hm$-matrix format here lead to storage costs of log-linear complexity. Asymptotic smoothness of the kernel and usage of admissibility criterion \cref{eq:strong_adm} or \cref{eq:weak_adm} guarantee that the submatrix ranks are bounded independently of the matrix size. However, we note that when the $\Hm$-matrix error is decreased at the same rate as the discretization error, the ranks increase logarithmically with $|I|$ and $|J|$.

\section{Uniform $\Hm$-matrix compression}\label{sec:4_alg}

We consider the compression of a regular $\Hm$-matrix into the $\UHm$-matrix format. The scheme presented here bears a close resemblance to the recompression of $\Hm^2$-matrices in \cite{Borm2007a,Borm2007b,Borm2009a} and $\DHm^2$-matrices in \cite{Borm2017a,Borm2020,Borm2023}. In the former, this also includes direct compression of general matrices and $\Hm$-matrices into the $\Hm^2$-matrix format. The focus of our method lies in the use of a uniform $\Hm$-matrix as an end goal, rather than an intermediate step towards an $\Hm^2$-matrix. In addition, we illustrate how compression and construction can be combined to build a $\UHm$-matrix approximation directly from a general matrix in log-linear time, along with error estimates and an analysis of computational complexity.

\subsection{Computation of optimal cluster bases}\label{ssec:4.1_probstat}
We first define the problem. We want to compress a matrix $A$ in $\Hm$-matrix format by transforming it to the $\UHm$-matrix format. \Cref{lem:storage_cost} showed that the gains are found in the factor $\csp$. As \cref{def:uhmat} states, two necessary ingredients are the row cluster basis $\mathcal{U}$ and column cluster basis $\mathcal{V}$. For reasons that will become clear later on, we restrict ourselves to \emph{orthogonal cluster bases}. We say that $\mathcal{U}$ is an orthogonal cluster basis if each basis matrix is orthogonal,
$$ U_\ct^*U_\ct = I,\quad \forall U_\ct \in \mathcal{U}.$$
If both $\mathcal{U}$ and $\mathcal{V}$ are orthogonal, each submatrix $\restr{A}{b}$ for $b=\ts$ can be projected orthogonally,
\begin{equation*}
\begin{aligned}
    U_\ct U_\ct^* \restr{A}{b} V_\cs V_\cs^* = U_\ct S_b V_\cs^*, && S_b:=U_\ct^*\restr{A}{b} V_\cs,
\end{aligned}
\end{equation*}
leading to the optimal coefficient matrix $S_b$ (w.r.t.\ the Frobenius norm) for the proposed $\mathcal{U}$ and $\mathcal{V}$.

To choose ``optimal'' cluster bases, we must first identify what it means for a cluster basis to be optimal. Considering a cluster $\ct\in\Lrc$ and its basis matrix $U_\ct\in\mathcal{U}$, the corresponding \emph{agglomeration matrix} $A_\ct$ of $A$ is the horizontal concatenation of all blocks involving $\tau$,
\begin{align*}
    A_\ct := \restr{A}{\ct\times\mathcal{F}(\ct)}\in\C^{\ct\times\mathcal{F}(\ct)}, \quad \mathcal{F}(\ct):=\dot{\bigcup}~{\row{\ct}} \\ \quad\text{s.t.}\quad \restr{A_{\tau}}{b} = \restr{A}{b}\quad\text{for all}~\cs\in\row{\ct}~\text{with}~b=\ts.
\end{align*}
We denote the corresponding submatrix of the uniform approximation $A^\UHm$ by $A^\UHm_\ct$. The error w.r.t.\ the spectral norm of this subblock can be decomposed as follows,
\begin{equation*}
\begin{aligned}
    \norm{A_\ct-A^\UHm_\ct}_2^2 &= \Big\|\sum_{\cs\in\row{\ct}} \left(\restr{A}{b} - U_\ct U_\ct^* \restr{A}{b} V_\cs V_\cs^*\right)\Big\|_2^2 \quad\quad\quad\quad\quad  (b=\ts) \\
    &= \Big\|\sum_{\cs\in\row{\ct}} \Big(\restr{A}{b} - U_\ct U_\ct^*\restr{A}{b} + U_\ct U_\ct^*\left(\restr{A}{b}- \restr{A}{b}V_\cs V_\cs^*\right)\Big)\Big\|_2^2 \\
    &= \Big\|A_\ct - U_\ct U_\ct^* A_\ct + \sum_{\cs\in\row{\ct}} U_\ct U_\ct^*\left(\restr{A}{b}- \restr{A}{b}V_\cs V_\cs^*\right)\Big\|_2^2 \\ 
    &\le \Big\|A_\ct - U_\ct U_\ct^* A_\ct\Big\|_2^2 + \sum_{\cs\in\row{\ct}}\Big\|U_\ct U_\ct^*\left(\restr{A}{b}- \restr{A}{b}V_\cs V_\cs^*\right)\Big\|_2^2 \\
    &\le \Big\|A_\ct - U_\ct U_\ct^* A_\ct\Big\|_2^2 + \sum_{\cs\in\row{\ct}}\Big\|\restr{A}{b}- \restr{A}{b}V_\cs V_\cs^*\Big\|_2^2.
\end{aligned}
\end{equation*}
Here, the error originating in the projection onto $U_\ct$ is decoupled from the error due to projections onto $V_\cs\in\mathcal{V}$. The same derivation holds in the Frobenius norm, in which case the first upper bound becomes an equality. The observation that orthogonal cluster bases result in decoupled errors is also made in the aforementioned $\Hm^2$- and $\mathcal{D}\Hm^2$-matrix compression schemes. Given specified tolerances $\epsilon_\ct$ for each $\ct\in\Lrc$, we define the optimal orthogonal cluster basis $\mathcal{U}$ by requiring
\begin{equation}\label{eq:optimal_basis}
U_\ct 
= \argmin_{Q\in \C^{\ct \times \ell_\ct}}~{\lt} \quad\text{s.t.}~\norm{A_\ct-Q Q^*A_\ct}\le\epsilon_\ct ~\text{and}~ Q^*Q=I, \quad \forall U_\ct \in \mathcal{U}.
\end{equation}
The analysis in \cref{sec:5_analysis} justifies this criterion (cf.\ infra).

The requirements on the cluster basis in the $\Hm^2$- and $\mathcal{D}\Hm^2$-matrix compression schemes is notably different from \cref{eq:optimal_basis}. In the former, a specified error $\epsilon_b$ is required on $\|\restr{A}{b}\|$ for each $b\in P^+_{\ItJ}$. This is necessary because the nestedness of the cluster basis results in more complicated error propagation. It is also the lack of nestedness in a uniform $\Hm$-matrix that makes the uniform compression considerably simpler.

For both the spectral and Frobenius norms, the optimal cluster basis is obtained from the left singular vectors of the truncated SVD of $A_\ct$. This is outlined in \cref{alg:clusterBasisConstr}. The column cluster basis $\mathcal{V}$ can be computed similarly from $A^*$.

\begin{algorithm2e}[t]
\SetKwInOut{Input}{input}
\SetKwInOut{Output}{output}
\SetKw{Init}{init}{}{}
\SetKwComment{Comment}{// }{}
\SetAlgoLined
\Input{Matrix $A\in\C^\ItJ$, admissible partition $P^+_\ItJ=\leaves^+(\tree_\ItJ)$ with row clusters $\Lrc$ and tolerances $\{\epsilon_\ct\}_{\ct\in\Lrc}$.}
\Output{Orthogonal row cluster basis $\mathcal{U}:=\{U_\ct\}_{\ct\in\Lrc}$}
\For{$\ct\in\Lrc$}{
    \Init $A_\tau \in \C^{\ct\times\mathcal{F}(\ct)}$ \DontPrintSemicolon \Comment*[l]{To be filled}
    \For{$\cs\in\row{\ct}$}{
        $A_{\tau}(:,\cs) \leftarrow A(\ct,\cs)$
    }
    $[U,\Sigma,V]\leftarrow\textsc{SVD}(A_\ct)$ \\
    Choose optimal rank $\lt$ according to $\epsilon_\ct$ \\
    $U_\ct \leftarrow U(:,1:\lt)$
}
 \caption{Construction of optimal orthogonal row cluster basis}
 \label{alg:clusterBasisConstr}
\end{algorithm2e}

\subsection{A log-linear algorithm for the compression of $\Hm$-matrices}
The compression as presented in \cref{alg:clusterBasisConstr} leads to a quadratic cost $\bigO(|I|\cdot|J|)$ when applied to all row and column clusters of $A\in\C^\ItJ$. This is due to the high cost of computing the SVD of the agglomeration matrices, of which the dimensions are proportional to the size of the corresponding cluster $\ct$ and its `far field' $\mathcal{F}(\ct)$. By exploiting the fact that $A$ is an $\Hm$-matrix, the cost can be improved to being log-linear in the size.

For agglomeration matrix $A_\ct$ of an $\Hm$-matrix $A$, we have
\begin{equation*}
    \restr{A_\ct}{b} = \restr{A}{b} = X_bY_b^* \quad\text{for all}~\cs\in\row{\ct}~\text{with}~b=\ts.
\end{equation*}
Consider the QR-decomposition of $Y_b\in\C^{\cs\times k_b}$. It can be cheaply computed because $Y_b$ is thin, giving us
\begin{equation*}
    \restr{A}{b} = X_bY_b^* = X_b R_{Y,b}^*Q_{Y,b}^*
\end{equation*}
with $Q_{Y,b}$ orthogonal. Carrying this out for all $\cs\in\row{\ct}$ and enumerating $\row{\ct}:=\{\cs_1,\dots,\cs_\nu\}$, $\nu:=\nu(\ct)=|\row{\ct}|$, the agglomeration matrix $A_\ct$ is decomposed:
\begin{equation}\label{eq:agglomeration_decomp}
\begin{aligned}
    A_\ct &= \left(\begin{array}{ccc}
                X_{b_1}Y_{b_1}^* & \dots & X_{b_\nu}Y_{b_\nu}^*  \\
            \end{array}\right) & (b_i=\ts_i) \\
        &= \left(\begin{array}{ccc}
                X_{b_1}R_{Y,b_1}^*Q_{Y,b_1}^* & \dots & X_{b_\nu}R_{Y,b_\nu}^*Q_{Y,b_\nu}^*  \\
            \end{array}\right) \\
        &= \left(\begin{array}{ccc}
                X_{b_1}R_{Y,b_1}^* & \dots & X_{b_\nu}R_{Y,b_\nu}^*  \\
            \end{array}\right)
            \left(\begin{array}{ccc}
                Q_{Y,b_1} && \\ & \ddots & \\ && Q_{Y,b_\nu}  \\
            \end{array}\right)^* 
        &=: X_\ct Q_\ct^* .
\end{aligned}
\end{equation}
Due to the orthogonality of $Q_{Y,b_i}$, $i=1,\dots,\nu$, the block-diagonal matrix $Q_\ct$ is also orthogonal. The SVD of the skinny matrix $X_\ct\in\C^{\ct\times \kt}$ is cheaper to compute as $\kt:=\sum_{i=1}^\nu k_{b_i}\le \nu \kmax \le \csp \kmax$ is bounded independently of the size of the matrix. Composition of this SVD with the adjoint of $Q_\ct$ results in the SVD of the full agglomeration matrix $A_\ct$. Thus, $X_\ct$ can be used in place of $A_\ct$ to find $U_\ct$.

\begin{remark}[Difference with $\Hm^2$-matrix compression]\label{rmk:comparison-H2}
In ${\mathcal H}^2$-recompression the cluster bases are nested, but for $\UHm$-matrices they are not. That means the definition of $X_\ct$ in the current setting can be non-recursive, i.e., it does not span multiple levels.
\end{remark}

After the computation of the rank-revealing SVDs, 
\begin{align*}
    X_\ct &= \svdU_1 \svdS_1 \svdV_1^*= \left(\begin{array}{cc}\hat{\svdU}_1 & \svdU_1^{(r)} \end{array}\right)\left(\begin{array}{cc}\hat{\svdS}_1 & \\ &\svdS_1^{(r)} \end{array}\right) \left(\begin{array}{cc}\hat{\svdV}_1 & \svdV_1^{(r)} \end{array}\right)^*, \\ 
    Y_\cs &= \svdU_2 \svdS_2 \svdV_2^* = \left(\begin{array}{cc}\hat{\svdU}_2 & \svdU_2^{(r)} \end{array}\right)\left(\begin{array}{cc}\hat{\svdS}_2 & \\ &\svdS_2^{(r)} \end{array}\right) \left(\begin{array}{cc}\hat{\svdV}_2 & \svdV_2^{(r)} \end{array}\right)^* \\ 
&\text{with}\quad \hat{\svdU}_1\in\C^{\ct\times \lt},~\lt \le \kt \quad\text{and}\quad \hat{\svdU}_2\in\C^{\cs\times \ls},~\ls \le \ks,
\end{align*} 
the optimal coefficient matrix $S_b$ for submatrix $\restr{A}{b}=X_bY_b^*$ with $b=\ts$ is obtained through left-multiplication of $X_b$ and $Y_b$ with the adjoint of $U_\ct:=\hat{\svdU}_1$ and $V_\cs:=\hat{\svdU}_2$ respectively. The result is $S_b:=S_{X,b}S_{Y,b}^* \in \C^{\lt \times \ls}$ where
\begin{align*}
    \hat{\svdU}_1^* X_b &= \hat{\svdU}_1^* X_b \left( R_{Y,b}^{*} {(R_{Y,b}^*)}^{-1} \right) & \hat{\svdU}_2^* Y_b &= \hat{\svdU}_2^* Y_b \left( R_{X,b}^{*} {(R_{X,b}^*)}^{-1} \right) \\
    &=\hat{\svdU}_1^* \left( X_b R_{Y,b}^{*} \right) {(R_{Y,b}^*)}^{-1}  & &=\hat{\svdU}_2^* \left( Y_b R_{X,b}^{*} \right) {(R_{X,b}^*)}^{-1}\\
    &= \hat{\svdU}_1^* \restr{X_\ct}{b} {(R_{Y,b}^*)}^{-1} & &= \hat{\svdU}_2^* \restr{Y_\cs}{b} {(R_{X,b}^*)}^{-1} \\
    &= \hat{\svdU}_1^* \svdU_1 \svdS_1 \restr{\svdV_1^*}{b} {(R_{Y,b}^*)}^{-1} & &= \hat{\svdU}_2^* \svdU_2 \svdS_2 \restr{\svdV_2^*}{b} {(R_{X,b}^*)}^{-1} \\
    &= \hat{\svdS}_1 \restr{\hat{\svdV}_1^*}{b} {(R_{Y,b}^*)}^{-1}=: S_{X,b}
    		& &= \hat{\svdS}_2 \restr{\hat{\svdV}_2^*}{b} {(R_{X,b}^*)}^{-1}=: S_{Y,b} .
\end{align*}
Using the last line above to compute the coefficient matrices is more efficient as it involves only small matrices. In contrast, the inner dimension of the product $\hat{\svdU}_1^* X_b$ is of the size $|\ct|$, even though $S_{X,b}$ is only of size $\lt\times\kb$, similarly for $S_{Y,b}$.
Note the slight abuse of notation in $\restr{X_\ct}{b}$ which here does not denote the restriction of $X_\ct$ to $b=\ts$ but rather the subset of columns corresponding to $b$ according to the definition of $X_\ct$.

Combining the steps of the QR-factorizations, the computation of the truncated SVD and the assembly of the coefficient matrices over all the row and column clusters, yields \cref{alg:unfCompression}. Note that we opt to keep the coefficient matrices in their factorized form $S_b:=S_{X,b}S_{Y,b}^*$, requiring $\kb(\lt + \ls)$ elements instead of $\lt \ls$. This choice does not significantly affect the total storage cost or eventual cost of the matrix-vector product as they are dominated by the costs due to the cluster basis matrices. Opting to store the coefficient matrices directly requires an additional iteration over all $b\in P^+_\ItJ$ to assemble $S_b$ from its factors.

\begin{algorithm2e}[tbhp]
\SetKwInOut{Input}{input}
\SetKwInOut{Output}{output}
\SetKw{Init}{init}{}{}
\SetKwComment{Comment}{// }{}
\SetAlgoLined
\Input{$\Hm$-matrix $A$ with admissible blocks $\{(X_b,Y_b)\}_{b\in P^+_\ItJ}$ and tolerances $\{\epsilon_\ct\}_{\ct\in\Lrc}$ and $\{\epsilon_\cs\}_{\ct\in\Lcc}$.}
\Output{Uniform $\Hm$-matrix approximation $A^\UHm$ with $\mathcal{U}=\{U_\ct\}_{\ct\in\Lrc}$, $\mathcal{V}=\{V_\cs\}_{\ct\in\Lcc}$ and $\{S_{X,b},S_{Y,b}\}_{b\in P^+_\ItJ}$}
\For{$\ct\in\Lrc$}{
    \Init $X_\ct \in \C^{\ct\times \kt}$ \DontPrintSemicolon \Comment*[l]{To be filled}
    $i \leftarrow 0$ \\
    \For{$\cs\in\row{\ct}$}{
        $b \leftarrow \ts$ \\
        $[Q,R] \leftarrow \textsc{QR}(Y_b)$ \\
        $R_{Y,b} \leftarrow R$ \\
        $X_\ct(:,i:i+k_b) \leftarrow X_b R_{Y,b}^*$ \\
        $i \leftarrow i+k_b$
    }
    $[U,\Sigma,V]\leftarrow\textsc{SVD}(X_\ct)$ \\
    Choose optimal rank $\lt$ according to $\epsilon_\ct$ \\
    $U_\ct \leftarrow U(:,1:\lt)$ \\
    $i \leftarrow 0$ \\
    \For{$\cs\in\row{\ct}$}{
        $b \leftarrow \ts$ \\
        $S_{X,b} \leftarrow \Sigma(1:\lt,1:\lt)V(i:i+k_b,1:\lt)^*$ \\
        $i \leftarrow i+k_b$ \\
        Solve lower-triangular system $S_{X,b}{(R_{Y,b}^{*})}^{-1}$ in-place
    }
}
\textbf{repeat \small{2--20} for} $\cs\in\Lcc$ to obtain $\mathcal{V}$ and $S_{Y,b}$
 \caption{Uniform compression of an $\Hm$-matrix}
 \label{alg:unfCompression}
\end{algorithm2e}

The operation count of the compression depends on the cost of the QR and SVD decompositions, of the triangular matrix-matrix product (TRMM), the diagonal scaling with the singular values and a  triangular (block) system solve (TRSM). These give the following operation counts:
\begin{itemize}
    \item QR takes no more than $C_{\mathrm{qr}}mn^2$ for an $m\times n$ matrix with $m>n$ \cite[\S5]{GolubLoan2013};
    \item the SVD of the same matrix 
    takes no more than $C_{\mathrm{svd}}^{(1)}mn^2 + C_{\mathrm{svd}}^{(2)}n^3$ \cite[\S8.6]{GolubLoan2013};
    \item both TRMM and TRSM \cite[\S3.1]{GolubLoan2013} require $mn^2$ operations for two matrices of sizes $m\times n$ and $n\times n$;
    \item diagonally scaling an $m\times n$ matrix takes $mn$ operations.
\end{itemize}

This leads to the following statement.
\begin{lemma}[Complexity of compression]\label{lem:comprComplexity}
    Given a $\Hm$-matrix $A$ with (block) cluster trees $\tree_I$, $\tree_J$ and $\tree_\ItJ$, and admissible partition $P^+_\ItJ=\leaves^+(\tree_\ItJ)$, there exist positive constants $C_{\mathrm{qlin}},C_{\mathrm{lin}}\in\R$ such that \cref{alg:unfCompression} requires no more than
    \begin{equation*}
        C_{\mathrm{qlin}} (\csp \kmax)^2 \big(\, L(\tree_I)|I| + L(\tree_J)|J| \,\big) ~+~ C_{\mathrm{lin}} (\csp \kmax)^3 \big(\, |I|+|J| \,\big)
    \end{equation*}
    operations, where $k_{\max}$ is the maximum block rank of $A$. Considering the BEM setting with $L(\tree_I)\sim\log|I|$ and $L(\tree_J)\sim\log|I|$, and $\csp$ and $\kmax$ bounded, this yields a complexity of $\bigO(|I|\log|I|+|J|\log|J|)$ for increasing dimension of the BEM matrix.
\end{lemma}
\begin{proof}
The number of operations needed for one row cluster $\ct\in\Lrc$ in \cref{alg:unfCompression} is
\begin{align*}
    & \sum_{\cs\in\row{\ct}}\left( C_{\mathrm{qr}}|\cs|k_{\ts}^2 + |\ct|k_{\ts}^2 \right) ~+~ C_{\mathrm{svd}}^{(1)}|\ct|{(\kt)}^2 + C_{\mathrm{svd}}^{(2)}{(\kt)}^3 \\&+~ \sum_{\cs\in\row{\ct}}\left( \lt k_\ts + \lt k_\ts^2 \right)
\end{align*}
which is bounded from above by
\begin{align*}
    & \sum_{\cs\in\row{\ct}}\left(C_{\mathrm{qr}}\kmax^2|\cs|\right) ~+~ \csp \kmax^2 |\ct| ~+~ C_{\mathrm{svd}}^{(1)}(\csp \kmax)^2|\ct| + C_{\mathrm{svd}}^{(2)}(\csp \kmax)^3 \\&+~ \csp \left(( \csp \kmax) \kmax + (\csp \kmax) \kmax^2 \right)
\end{align*}
with $\kt=\sum_{\cs\in\row{\ct}}k_{\ts}\le \csp \kmax$. Summing over all $\ct\in\Lrc$ and $\cs\in\Lcc $ and subsequently employing 
$\sum_{\ct\in\Lrc}\sum_{\cs\in\row{\ct}}=\sum_{\ts\in P^+_\ItJ} = \sum_{\cs\in\Lcc}\sum_{\ct\in\col{\cs}}$ and upper bounds $\sum_{\ct\in\Lrc}|\ct|\le L(\tree_I)|I|$ and $\sum_{\ct\in\Lrc}1 \le 2|I|/n_\mathrm{min}$, yields
\begin{multline*}
    \left( C_{\mathrm{qr}} + \csp C_{\mathrm{svd}}^{(1)} + 1 \right) \csp \kmax^2 \big(\, L(\tree_I)|I| + L(\tree_J)|J| \,\big) \\ + \left( \csp \kmax C_{\mathrm{svd}}^{(2)} + \kmax + 1 \right) \frac{2\csp^2 \kmax^2}{n_\mathrm{min}} \big(\, |I|+|J| \,\big).
\end{multline*}
Choosing $C_{\mathrm{qlin}}:=C_{\mathrm{svd}}^{(1)} + \frac{C_{\mathrm{qr}}+1}{\csp}$ and $C_{\mathrm{lin}}:=\frac{2}{n_\mathrm{min}}\left( C_{\mathrm{svd}}^{(2)} + \frac{1}{\csp} + \frac{1}{\csp \kmax}\right)$
proves the lemma.
\end{proof}

\subsection{Direct $\UHm$-matrix construction from matrix entries}
The algorithms as presented thus far, applied to an $\Hm$-matrix approximation $A^\Hm$, result in both matrices $A^\Hm$ and $A^\UHm$ being in memory simultaneously, which may present a bottleneck. Yet, the construction of a $\Hm$-matrix and its compression into $\UHm$-matrix format can also be combined, starting from entries of the original matrix $A$.

A straightforward approach to realize the direct construction of a $\UHm$-matrix is to incorporate the assembly of the admissible blocks of the $\Hm$-matrix into \cref{alg:unfCompression}. The factorization $X_bY_b^*$ of each block $b=\ts$ is computed when it is first needed. Once $X_b$ is used to assemble $X_\ct$ it can be thrown away and similarly for $Y_b$ and $Y_\cs$. 

The sets of row and column clusters $\Lrc$ and $\Lcc$ are sorted simultaneously from large to small for two reasons:
\begin{itemize}
    \item Applying the uniform compression to each cluster in this order ensures that larger clusters, and thus also larger blocks, are processed early.
    \item $\restr{A}{b}\approx X_bY_b^*$ is computed when either the row or column cluster is being processed. However, the factorization can only be completely thrown away once both clusters are fully processed. As the row and column cluster of a block can be expected to be of similar size, the clusters will be sorted close together, reducing the time $X_b$ or $Y_b$ spends in memory.
\end{itemize}
Subsequently, the maximum amount of memory in use over the course of the construction is expected to be close to the memory cost of the complete $\UHm$-matrix.

\Cref{alg:unfConstruction-helper} gives the procedure to compute the cluster basis matrix $U_\ct$ of a row cluster $\ct$ and the corresponding coefficient matrices $\{S_{X,\ts}\}_\ts$. Tolerances for the truncation of $X_\ct$'s SVD and for the ACA approximations of the blocks are provided as inputs. It is assumed that the ACA algorithm incorporates \emph{algebraic recompression} into $U_b\Sigma_bV_b^*$ (see \cref{ssec:2.3_ACA}) such that the agglomeration matrices can be defined as $X_{\ct|b}:=U_b\Sigma_b$ and $Y_{\cs|b}:=V_b\Sigma_b$.\footnote{Using $X_{\ct|b}:=U_b\Sigma_b$ differs from the standard definition of $X_{\ct|b}:=X_bR_{Y,b}^*$ in \cref{eq:agglomeration_decomp}. Multiplying with $\Sigma_b$ acts as a scaling -- similarly to multiplying with $R_{Y,b}^*$ -- to enable the error analysis later on.}

The inputs $\{(U_\ts,\Sigma_\ts)\}_{\cs\in\row{\ct}}$ to \cref{alg:unfConstruction-helper} allow to provide factors of blocks that were already treated by their column cluster. Conversely, outputs $\{(V_\ts,$ $\Sigma_\ts)\}_{\cs\in\row{\ct}}$ return factors that are still needed. Line 4 checks whether the inputs are provided on a block-by-block basis. Lines 8 and 17--19 ensure that factors are thrown away once they are fully processed. 

\begin{algorithm2e}[tbhp]
\SetKwInOut{Input}{input}
\SetKwInOut{Output}{output}
\SetKw{Init}{init}{}{}
\SetKwComment{Comment}{// }{}
\SetAlgoLined
\nonl\textbf{procedure} \textsc{BuildCluster} \\
\Input{Matrix $A : \ItJ \to \C : (i,j) \mapsto a_{ij}$, row cluster $\ct$, tolerances $\epsilon_\ct$ and $(\epsilon_\ts)_{\cs\in\row{\ct}}$, and $\{(U_\ts,\Sigma_\ts)\}_{\cs\in\row{\ct}}$}
\Output{Basis matrix $U_\ct$, coefficients $\{S_{X,\ts}\}_{\cs\in\row{\ct}}$ and $\{(V_\ts,\Sigma_\ts)\}_{\cs\in\row{\ct}}$}
\Init $X_\ct \in \C^{\ct\times \kt}$ \DontPrintSemicolon \Comment*[l]{To be filled}
$i \leftarrow 0$ \\
\For{$\cs\in\row{\ct}$}{
    \If{$(U_\ts,\Sigma_\ts) = \emptyset$}
    {
        $[U_\ts,\Sigma_\ts,V_\ts] \leftarrow \textsc{ACA}(\restr{A}{\ts},\epsilon_\ts)$
    }
    $X_\ct(:,i:i+k_\ts) \leftarrow U_\ts \Sigma_\ts$ \\
    $U_\ts \leftarrow \emptyset$ \\
    $i \leftarrow i+k_\ts$
}
$[U,\Sigma,V]\leftarrow\textsc{SVD}(X_\ct)$ \\
Choose optimal rank $\lt$ according to $\epsilon_\ct$ \\
$U_\ct \leftarrow U(:,1:\lt)$ \\
$i \leftarrow 0$ \\
\For{$\cs\in\row{\ct}$}{
    $S_{X,\ts} \leftarrow \Sigma(1:\lt,1:\lt)V(i:i+k_\ts,1:\lt)^*\Sigma_\ts^{-1/2}$ \\
    \If{$V_\ts=\emptyset$}
    {
        $\Sigma_\ts \leftarrow \emptyset$
    }
    $i \leftarrow i+k_\ts$ \\
}
 \caption{Construction of a row cluster of a $\UHm$-matrix}
 \label{alg:unfConstruction-helper}
\end{algorithm2e}

\begin{remark}[Alternative rank-revealing decompositions]\label{rmk:rr-decomp}
	\Cref{alg:unfConstruction-helper} uses the standard truncated SVD to determine the optimal cluster basis matrix and corresponding rank, as first proposed in \cref{ssec:4.1_probstat}. However, other rank-revealing decompositions are also valid. This can reduce the compression time at the expense of optimality of the ranks. One alternative is a rank-revealing QR-decomposition; column-pivoted QR with early-exit has $\bigO(mnk)$ complexity for an $m$-by-$n$ matrix of rank $k$ while the SVD has $\bigO(mn^2)$ complexity regardless of the obtained rank. Another option is a low-rank decomposition based on randomized linear algebra, see \cite[\S4]{Murray2023} and references therein.
\end{remark}

\Cref{alg:unfConstruction} details the complete construction of a $\UHm$-matrix by applying the procedure from \cref{alg:unfConstruction-helper} to all clusters.\footnote{The construction of the inadmissible blocks is left out, as this is just evaluating and storing all the corresponding submatrices.} For the column clusters in $\Lcc$, the adjoint of matrix $A$ is supplied to the subroutine and arguments are swapped appropriately.

\begin{algorithm2e}[ht]
\SetKwInOut{Input}{input}
\SetKwInOut{Output}{output}
\SetKw{Init}{init}{}{}
\SetAlgoLined
\Input{Matrix $A : \ItJ \to \C : (i,j) \mapsto a_{ij}$, partition $P^+_\ItJ$ and tolerances $(\epsilon_\ct)_{\ct\in\Lrc}$, $(\epsilon_\cs)_{\cs\in\Lcc}$ and $(\epsilon_b)_{b\in P^+_\ItJ}$}
\Output{$\UHm$-matrix $A^\UHm$ with $\mathcal{U}=\{U_\ct\}_{\ct\in\Lrc}$, $\mathcal{V}=\{V_\cs\}_{\cs\in\Lcc}$ and $\{S_{X,b},S_{Y,b}\}_{b\in P^+_\ItJ}$}
$\mathcal{L}' \leftarrow \textsc{sort}(\Lrc,\Lcc)$ \\
$\{(U_\ts,\Sigma_\ts,V_\ts)\}_{\ts\in P^+_\ItJ} \leftarrow \emptyset$ \\
\For{$\ct\in\mathcal{L}'$}{
    \If{$\ct \in \Lrc$}{
        $\left[U_\ct, \{S_{X,\ts}\}_{\cs\in\row{\ct}},\{(V_\ts,\Sigma_\ts)\}_{\cs\in\row{\ct}}\right] \leftarrow$ \\ \Indp $\textsc{BuildCluster}(A,\ct,\epsilon_\ct,(\epsilon_\ts)_{\cs\in\row{\ct}},\{(U_\ts,\Sigma_\ts)\}_{\cs\in\row{\ct}})$
    }\Else{
        $\left[V_\ct, \{S_{Y,\st}\}_{\cs\in\col{\ct}},\{(U_\st,\Sigma_\st)\}_{\cs\in\col{\ct}}\right] \leftarrow$ \\ \Indp $\textsc{BuildCluster}(A^*,\ct,\epsilon_\ct,(\epsilon_\st)_{\cs\in\col{\ct}},\{(V_\st,\Sigma_\st)\}_{\cs\in\col{\ct}})$
    }
}
 \caption{Construction of a uniform $\Hm$-matrix}
 \label{alg:unfConstruction}
\end{algorithm2e}

\begin{remark}[Cost of initial $\Hm$-matrix construction]\label{rmk:cost_hmat_constr}
    The computational cost of the initial $\Hm$-matrix construction has a similar operation count as that of \cref{lem:comprComplexity} but with other constants. In Galerkin BEM, the most expensive part is the sampling of the matrix elements. This contributes $C_{\mathrm{aca}}^{(1)}C_\mathrm{elem}\csp \kmax \big(\, L(\tree_I)|I| + L(\tree_J)|J| \,\big)$ operations. While the uniform compression contains an additional factor of $\csp \kmax$, $C_{\mathrm{aca}}^{(1)}C_\mathrm{elem}$ is generally much larger than $C_{\mathrm{qlin}}$. Thus, it may be expected that the relative additional cost of uniform compression is small.
\end{remark}

\begin{remark}[Symmetry]\label{rmk:symmetry}
    If the matrix $A\in\C^\ItJ$ is symmetric (or hermitian), and naturally $I=J$, it is customary to choose $\tree_I=\tree_J$. One can then define a single set of clusters 
    \begin{equation}\label{eq:symmetry}
        \Lrc(\tree_\ItI):=\{\ct\in\tree_I : \exists\cs\in\tree_I~~\text{s.t.\ }~\ts\in\leaves^+(\tree_\ItI) \vee \cs\times\ct\in\leaves^+(\tree_\ItI)  \}
    \end{equation}
    and combine the corresponding $\{X_{\ts}\}_{\cs\in\row{\ct}}$ and $\{Y_{\cs\times\ct}\}_{\cs\in\col{\ct}}$ in one single $X_\ct$. If block cluster tree $\tree_\ItI$ is itself symmetric, only admissible blocks below or above the diagonal have to be computed to be used in place of their counterpart above (or below) the diagonal. Even if $A$ is neither symmetric nor hermitian but $I=J$, it is still an option to join the clusters into one set \cref{eq:symmetry} if $\tree_I=\tree_J$.
\end{remark}

\subsection{Parallel construction and matrix-vector product}\label{ssec:4.4_parallel}
\Cref{alg:unfConstruction} provides a sequential program to construct a $\UHm$-matrix. However, to utilize current computer architectures more effectively, a parallel implementation is necessary. In this section we describe a parallel implementation assuming shared-memory parallelism. Following additions/modifications are applied:
\begin{itemize}
    \item the for-loop from lines 3-12 in \cref{alg:unfConstruction} is distributed over the threads, taking load balancing into account;
    \item locks are initialized at the start of the program, one for each admissible block;
    \item to execute lines 4-6 of \cref{alg:unfConstruction-helper}, a thread needs to acquire the corresponding lock of the admissible block. If another thread has already acquired the lock, the current thread waits until the lock is released;
    \item once released, the factorization of the block is guaranteed to have been computed by the other thread. Thus, the thread can continue at line 7.
\end{itemize}
In theory, the locking mechanism can slow down the parallel computation of the block factorizations by a factor of two in the worst case, namely when for each block, either the column cluster has to wait on the row cluster or vice versa. While the numerical experiments suggest that this is not an issue in the current setup, many more CPU cores, relative to the number of blocks and clusters, may result in significant \emph{thread blocking}.

The matrix-vector products also benefit from a parallel implementation. \Cref{alg:hmatvec,alg:uhmatvec} provide a parallel matrix-vector product for the $\Hm$-matrix and $\UHm$-matrix in a shared-memory setting.

\begin{algorithm2e}[tbhp]
\SetKwInOut{Input}{input}
\SetKwInOut{Output}{output}
\SetKw{Init}{init}{}{}
\SetAlgoLined
\Input{Hierarchical matrix $A$, vector $\vv\in\C^{J}$ and number of threads $p$}
\Output{$\vw=A\vv\in\C^{I}$}
$\{\mathcal{P}^{(q)}(P_\ItJ^+)\}_{q=0}^p\gets P_\ItJ^+$ (Distribute admissible blocks) \\
$\{\mathcal{P}^{(q)}(P_\ItJ^-)\}_{q=0}^p\gets P_\ItJ^-$ (Distribute dense blocks) \\
\SetKwBlock{Parfor}{parfor}{end}
\Parfor($0<q<p$){
    \Init $\vw^{(q)}=\mathbf{0}\in\C^{I}$ \\
    \For{$\ts\in\mathcal{P}^{(q)}(P_\ItJ^+)$}{
        $\restr{\vw^{(q)}}{\ct} \gets \restr{\vw^{(q)}}{\ct} + X_\ts Y_\ts^* \restr{\vv}{\cs}$
    }
    \For{$\ts\in\mathcal{P}^{(q)}(P_\ItJ^-)$}{
        $\restr{\vw^{(q)}}{\ct} \gets \restr{\vw^{(q)}}{\ct} + \restr{A}{\ts}\restr{\vv}{\cs}$
    }
}
$\vw \gets \sum \vw^{(q)}$ (parallel reduction)
    
 \caption{Parallel $\Hm$-matrix-vector product}
 \label{alg:hmatvec}
\end{algorithm2e}

\begin{algorithm2e}[tbhp]
\SetKwInOut{Input}{input}
\SetKwInOut{Output}{output}
\SetKw{Init}{init}{}{}
\SetKw{Sync}{sync}{}{}
\SetAlgoLined
\Input{$\UHm$-matrix $A$, vector $\vv\in\C^{J}$ and number of threads $p$}
\Output{$\vw=A\vv\in\C^{I}$}
$\{\mathcal{P}^{(q)}(\Lrc)\}_{q=0}^p\gets \Lrc$ (Distribute row clusters) \\
$\{\mathcal{P}^{(q)}(\Lcc)\}_{q=0}^p\gets \Lcc$ (Distribute column clusters) \\
$\{\mathcal{P}^{(q)}(P_\ItJ^-)\}_{q=0}^p\gets P_\ItJ^-$ (Distribute dense blocks) \\
\SetKwBlock{Parfor}{parfor}{end}
\Parfor($0<q<p$){
    \For{$\cs\in\mathcal{P}^{(q)}(\Lcc)$}{
        $\vu_\cs \gets V_\cs^* \restr{\vv}{\cs}$ \\
        \For{$\ct\in\col{\cs}$}{
            $\vu_\ts \gets S_{Y,\ts}^* \vu_\cs$
        }
    }
}
\SetKwBlock{Parfor}{parfor}{end}
\Parfor($0<q<p$){
    \Init $\vw^{(q)}=\mathbf{0}\in\C^{I}$ \\
    \For{$\ct\in\mathcal{P}^{(q)}(\Lrc)$}{
        \Init $\vu_\ct = \mathbf{0} \in \C^{\ct}$ \\
        \For{$\ct\in\col{\cs}$}{
            $\vu_\ct \gets \vu_\ct + S_{X,\ts} \vu_\ts$
        }
        $\restr{\vw^{(q)}}{\ct} \gets \restr{\vw^{(q)}}{\ct} + U_\ct \vu_\ct$
    }
    \For{$\ts\in\mathcal{P}^{(q)}(P_\ItJ^-)$}{
        $\restr{\vw^{(q)}}{\ct} \gets \restr{\vw^{(q)}}{\ct} + \restr{A}{\ts}\restr{\vv}{\cs}$
    }
}
$\vw \gets \sum \vw^{(q)}$ (parallel reduction)
    
 \caption{Parallel $\UHm$-matrix-vector product}
 \label{alg:uhmatvec}
\end{algorithm2e}

For the $\Hm$-matrix version, both the admissible and dense blocks are distributed over the threads. We assume that dynamic load balancing is applied to both sets, where they are ordered from large to small. This means that distribution of the blocks is done during the execution of the parallel loop. An improvement would be to do static load balancing as a preprocessing step as the time each block takes can be estimated based on the computational cost. We refer to \cite{BebendorfKriemann2005,Hoshino2022} for more details on parallelization improvements.

The $\UHm$-matrix-vector product follows the same idea. Here, the admissible blocks are not distributed, but the row and column clusters are. The multiplication with the admissible blocks is now achieved in a two-step process. First the appropriate segment of the input vector is multiplied by $V_\cs^*$ after which it is multiplied by $S_{Y,\ts}^*$ for each block $\ts$. This result is stored in a temporary vector $\vu_\ts \in \C^{k_{\ts}}$. In the second step, this vector is multiplied by $S_{X,\ts}$, summed with contributions of the other blocks with the same row cluster and multiplied by $U_\ct$. The use of $\vu_\ts$ results in a minimal, implicit communication step between the thread that processes column cluster $\cs$ and the thread that processes row cluster $\ct$.

In both matrix-vector products each matrix element stored in the $\Hm$- or $\UHm$-matrix representation is used at most twice, indicating $\bigO(N\log N)$ complexity. This also means that the relative memory reduction in applying the uniform compression may directly result in similar relative speed-up of the matrix-vector product.

\section{Error analysis}\label{sec:5_analysis}

The compression described in the previous section ensures local, block-wise errors on the agglomeration matrices $A_\ct$ and $A^*_\cs$. Still, it is also important to guarantee a global error on the complete matrix $A$. The most convenient norm for this analysis is the Frobenius norm:
\begin{equation*}
\begin{aligned}
    \norm{A-A^{\UHm}}_F^2 &= \sum_{b=\ts\in P_\ItJ^+}\norm{\restr{A}{b}-U_\ct U_\ct^*\restr{A}{b}V_\cs V_\cs^*}_F^2 \\
    &= \sum_{b=\ts\in P_\ItJ^+}\norm{\restr{A}{b} - U_\ct U_\ct^*\restr{A}{b} + U_\ct U_\ct^*\restr{A}{b} - U_\ct U_\ct^*\restr{A}{b}V_\cs V_\cs^*}_F^2 \\
    &= \sum_{b=\ts\in P_\ItJ^+}\norm{\restr{A}{b} - U_\ct U_\ct^*\restr{A}{b}}_F^2 + \norm{U_\ct U_\ct^*(\restr{A}{b} - \restr{A}{b}V_\cs V_\cs^*)}_F^2 \\
    &\le \sum_{b=\ts\in P_\ItJ^+}\norm{\restr{A}{b} - U_\ct U_\ct^*\restr{A}{b}}_F^2 + \norm{\restr{A}{b} - \restr{A}{b}V_\cs V_\cs^*}_F^2 \\
    &= \sum_{\ct\in\Lrc}\sum_{\cs\in\row{\ct}}\norm{\restr{A}{b} - U_\ct U_\ct^*\restr{A}{b}}_F^2 + \sum_{\cs\in\Lcc}\sum_{\ct\in\col{\cs}}\norm{\restr{A}{b} - \restr{A}{b}V_\cs V_\cs^*}_F^2 \\
    &= \sum_{\ct\in\Lrc}\norm{A_\ct - U_\ct U_\ct^*A_\ct}_F^2 + \sum_{\cs\in\Lcc}\norm{A^*_\cs -V_\cs V_\cs^*A^*_\cs}_F^2.
\end{aligned}
\end{equation*}
The block-wise errors directly translate into the global error.

\Cref{thm:spectral_error} below provides a bound for the spectral norm. The analysis here is similar to the one for general norms in \cite[Lemma 6.26]{Hackbusch2015book}. The result of the theorem can also be extended to other norms when orthogonality in such norm is assumed.
\begin{theorem}[Global error of the $\UHm$-matrix compression]\label{thm:spectral_error}
    Given an $\Hm$-matrix $A\in\C^\ItJ$ and the approximation $A^\UHm$ computed by \cref{alg:unfCompression}, the global and local errors in the spectral norm satisfy
    \begin{equation}\label{eq:spectral_error}
        \norm{A-A^\UHm}_2 \le \sqrt{2}\sqrt{\sum_{\ct\in\Lrc}\norm{A_\ct - U_\ct U_\ct^*A_\ct}_2^2 + \sum_{\cs\in\Lcc}\norm{A^*_\cs -V_\cs V_\cs^*A^*_\cs}_2^2}\,.
    \end{equation}
    Suppose the local errors are scaled according to
    \begin{equation}\label{eq:spectral_error_local}
        \norm{A_\ct - U_\ct U_\ct^*A_\ct}_2 = \dfrac{\epsilon}{2}\sqrt{\dfrac{|\ct|\cdot|\mathcal{F}(\ct)|}{|I|\cdot|J|}} \;\text{and}\; \norm{A^*_\cs -V_\cs V_\cs^*A^*_\cs}_2 = \dfrac{\epsilon}{2}\sqrt{\dfrac{|\cs|\cdot|\mathcal{F}(\cs)|}{|I|\cdot|J|}}
    \end{equation}
    for all $\ct\in\Lrc$, $\cs\in\Lcc$. The global error then becomes
    \begin{equation}\label{eq:spectral_error_global}
        \norm{A-A^\UHm}_2 \le \epsilon.
    \end{equation}
\end{theorem}
\begin{proof}
The norm $\norm{A-A^\UHm}_2$ is the supremum of the absolute value of the scalar product $\langle (A-A^\UHm)\vu,\vv\rangle$ over all $\vu\in\C^J$, $\vv\in\C^I$ with $\norm{\vu}_2=\norm{\vv}_2=1$. In the following, the subscript is omitted in the norms for readability, i.e.\ $\norm{\cdot}:=\norm{\cdot}_2$ and the definition $b:=\ts$ is not repeated. Using the scalar products between the restrictions $\vu_\cs:=\restr{\vu}{\cs}\in\C^\cs$ and $\vv_\ct:=\restr{\vv}{\ct}\in\C^\ct$, the global scalar product is
\begin{equation*}
    \left\langle (A-A^\UHm)\vu,\vv \right\rangle = \sum_{b\in P^+_\ItJ}\big\langle (\restr{A}{b}-U_\ct U_\ct^* \restr{A}{b} V_\cs V_\cs^*)\vu_\cs, \vv_\ct \big\rangle.
\end{equation*}
Decomposing the error in each block as
\begin{equation*}
\begin{aligned}
    \restr{A}{b} - U_\ct U_\ct^* \restr{A}{b} V_\cs V_\cs^* &= \restr{A}{b} - U_\ct U_\ct^* \restr{A}{b} + U_\ct U_\ct^* \restr{A}{b} - U_\ct U_\ct^* \restr{A}{b} V_\cs V_\cs^* \\ &= (\restr{A}{b} - U_\ct U_\ct^* \restr{A}{b}) + U_\ct U_\ct^*(\restr{A}{b} - \restr{A}{b} V_\cs V_\cs^*)
\end{aligned}
\end{equation*}
and filling it back into the scalar products gives
\begin{equation*}
\begin{aligned}
    \left\langle (A-A^\UHm)\vu,\vv \right\rangle &= \sum_{b\in P^+_\ItJ}\big\langle (\restr{A}{b} - U_\ct U_\ct^* \restr{A}{b})\vu_\cs, \vv_\ct \big\rangle \\
    &+ \sum_{b\in P^+_\ItJ}\big\langle U_\ct U_\ct^*(\restr{A}{b} - \restr{A}{b} V_\cs V_\cs^*)\vu_\cs, \vv_\ct \big\rangle.
\end{aligned}
\end{equation*}
The first sum works out to
\begin{equation*}
\begin{aligned}
    \sum_{b\in P^+_\ItJ}\big\langle (\restr{A}{b} - U_\ct U_\ct^* \restr{A}{b})\vu_\cs, \vv_\ct \big\rangle &= \sum_{\ct\in\Lrc}\sum_{\cs\in\row{\ct}}\big\langle (\restr{A}{b} - U_\ct U_\ct^* \restr{A}{b})\vu_\cs, \vv_\ct \big\rangle \\
    &= \sum_{\ct\in\Lrc} \big\langle (A_\ct - U_\ct U_\ct^* A_\ct)\vu_{\mathcal{F}(\ct)}, \vv_\ct \big\rangle \\
\end{aligned}
\end{equation*}
and the second to
\begin{equation*}
\begin{aligned}
    \sum_{b\in P^+_\ItJ}\big\langle U_\ct U_\ct^*(\restr{A}{b} - \restr{A}{b} V_\cs V_\cs^*)\vu_\cs, \vv_\ct \big\rangle &= \sum_{b\in P^+_\ItJ}\big\langle \vu_\cs, (\restr{A}{b}^* - V_\cs V_\cs^*\restr{A}{b}^*)U_\ct U_\ct^*\vv_\ct \big\rangle \\
    &= \sum_{\cs\in\Lcc}\sum_{\ct\in\col{\cs}}\big\langle \vu_\cs, (\restr{A}{b}^* - V_\cs V_\cs^*\restr{A}{b}^*)U_\ct U_\ct^*\vv_\ct \big\rangle \\
    &= \sum_{\cs\in\Lcc}\big\langle \vu_\cs, (A^*_\cs - V_\cs V_\cs^*A^*_\cs)\mathbf{U}_\cs \mathbf{U}_\cs^*\vv_{\mathcal{F}(\cs)} \big\rangle \\
\end{aligned}
\end{equation*}
where $\mathbf{U}_\cs$ is defined as the orthogonal matrix
$$ \mathbf{U}_\cs := \left(\begin{array}{ccc}
    U_{\ct_1} & & \\ & \ddots & \\ && U_{\ct_\nu}
\end{array}\right) \quad\text{with}\quad \col{\cs}:=\{\ct_1,\dots,\ct_\nu\},~\nu:=\nu(\cs)=|\col{\cs}|.$$
At this point we take absolute values such that
\begin{equation*}
\begin{aligned}
    \left|\left\langle (A-A^\UHm)\vu,\vv \right\rangle\right| &
    \le \sum_{\ct\in\Lrc} \left|\left\langle (A_\ct - U_\ct U_\ct^* A_\ct)\vu_{\mathcal{F}(\ct)}, \vv_\ct \right\rangle\right| \\
        &\quad\quad\quad+ \sum_{\cs\in\Lcc}\left|\left\langle \vu_\cs, (A^*_\cs - V_\cs V_\cs^*A^*_\cs)\mathbf{U}_\cs \mathbf{U}_\cs^*\vv_{\mathcal{F}(\cs)} \right\rangle\right| \\
    &\le \sum_{\ct\in\Lrc} \norm{A_\ct - U_\ct U_\ct^* A_\ct}\norm{\vu_{\mathcal{F}(\ct)}}\norm{\vv_\ct} \\
        &\quad\quad\quad+ \sum_{\cs\in\Lcc} \norm{(A^*_\cs - V_\cs V_\cs^*A^*_\cs)\mathbf{U}_\cs \mathbf{U}_\cs^*}\norm{\vu_\cs}\norm{\vv_{\mathcal{F}(\cs)}} \\
    &\le \sum_{\ct\in\Lrc} \norm{A_\ct - U_\ct U_\ct^* A_\ct}\norm{\vu_{\mathcal{F}(\ct)}}\norm{\vv_\ct} \\
        &\quad\quad\quad+ \sum_{\cs\in\Lcc} \norm{A^*_\cs - V_\cs V_\cs^*A^*_\cs}\norm{\vu_\cs}\norm{\vv_{\mathcal{F}(\cs)}}.
\end{aligned}
\end{equation*}
Next we apply the Schwarz inequality
\begin{multline*}
    \left|\left\langle (A-A^\UHm)\vu,\vv \right\rangle\right|^2 \le \left[ \sum_{\ct\in\Lrc} \norm{A_\ct - U_\ct U_\ct^* A_\ct}^2 + \sum_{\cs\in\Lcc} \norm{A^*_\cs - V_\cs V_\cs^*A^*_\cs}^2\right] \\ 
    \cdot\left[ \sum_{\ct\in\Lrc}\norm{\vu_{\mathcal{F}(\ct)}}^2\norm{\vv_\ct}^2 + \sum_{\cs\in\Lcc}\norm{\vu_\cs}^2\norm{\vv_{\mathcal{F}(\cs)}}^2 \right]
\end{multline*}
where
\begin{equation*}
\begin{aligned}
&\sum_{\ct\in\Lrc}\norm{\vu_{\mathcal{F}(\ct)}}^2\norm{\vv_\ct}^2 + \sum_{\cs\in\Lcc}\norm{\vu_\cs}^2\norm{\vv_{\mathcal{F}(\cs)}}^2 \\
	&\quad\quad\quad\quad\quad\le 2\sum_{b\in P_\ItJ}\norm{\vu_\cs}^2\norm{\vv_\ct}^2  =2\sum_{b\in P_\ItJ}\left[\sum_{\cs'\in\leaves(\cs)}\norm{\vu_{\cs'}}^2\right] \left[\sum_{\ct'\in\leaves(\ct)}\norm{\vv_{\ct'}}^2\right] \\
    &\quad\quad\quad\quad\quad=2\left[\sum_{\cs'\in\leaves(\tree_J)}\norm{\vu_{\cs'}}^2\right] \left[\sum_{\ct'\in\leaves(\tree_I)}\norm{\vv_{\ct'}}^2\right]  = 2\norm{\vu}^2\norm{\vv}^2.
\end{aligned}
\end{equation*}
Here we used $\leaves(\ct):=\leaves(\tree_\ct)$ to indicate the leaves of the subtree of $\tree_I$ with $\ct$ as root and similarly for $\leaves(\cs)$.
This eventually shows
\begin{equation*}
    \left|\left\langle (A-A^\UHm)\vu,\vv \right\rangle\right| = \sqrt{2}\sqrt{\sum_{\ct\in\Lrc} \norm{A_\ct - U_\ct U_\ct^* A_\ct}^2 + \sum_{\cs\in\Lcc} \norm{A^*_\cs - V_\cs V_\cs^*A^*_\cs}^2}\norm{\vu}\norm{\vv},
\end{equation*}
proving the inequality in \cref{eq:spectral_error}. The estimate \cref{eq:spectral_error_global} follows from \cref{eq:spectral_error_local} since $\{\ct\times\mathcal{F}(\ct)\}_{\ct\in\Lrc}$ and $\{\mathcal{F}(\cs)\times\cs\}_{\cs\in\Lcc}$ are collections of disjoint subsets of $\ItJ$.
\end{proof}
\begin{remark}
    Comparing the result of \cref{thm:spectral_error} with \cite[Remark 6.27]{Hackbusch2015book}, the former has an extra factor $1/2$ in the choice of local errors \cref{eq:spectral_error_local} because the errors due to projection onto the row cluster basis and column cluster basis are treated simultaneously. In practice, these projection errors as well as the error of the initial $\Hm$-matrix approximation should be taken into account as a whole.
\end{remark}

\section{Numerical experiments}\label{sec:6_exp}

Our aim is to assess the usefulness of the uniform $\Hm$-matrix format. To that end, we numerically investigate the compression properties of the $\UHm$-matrix format using \cref{alg:unfConstruction} proposed at the end of \cref{sec:4_alg}. The experiments are performed using the open-source \CC~package \href{https://gitlab.kuleuven.be/numa/software/beachpack}{BEACHpack}.\footnote{The code used to produce the results of the numerical experiments can be found at \href{https://gitlab.kuleuven.be/numa/software/beachpack/-/tree/kb-exp-paper-uhmatrix?ref_type=tags}{gitlab.kuleuven.be/numa/software/beachpack/-/tree/kb-exp-paper-uhmatrix}.}

\subsection{Experimental set-up}\label{ssec:6.1_setup}
The matrices under consideration in the numerical experiments originate from the boundary integral operator with kernel function 
\begin{equation*}
    g(\x,\y) = \dfrac{\mathrm{e}^{i \kappa {\|\x-\y\|}_2}}{4\pi{\|\x-\y\|}_2}.
\end{equation*}
This corresponds to the single-layer boundary operator of the 3D Helmholtz equation with wavenumber $\kappa$. The operator is discretized using discontinuous Galerkin BEM on a mesh of $N_t$ triangles, approximating the boundary $\Gamma$. In this case, $\mathscr{V}_N=\mathscr{W}^*_M:=\mathrm{span}\{\phi_i^{\ell}~|~(i,\ell) \in \{1,\dots,N_t\} \times \{1,2,3\} \}$. The $i$ in basis function $\phi_i^{\ell}$ refers to the only mesh triangle in which it is non-zero. Triangle $i$ has three such linear basis functions, each attaining the value $1$ in one of the three vertices and zero in the others. The function $\phi_i^{\ell}$ attains $1$ in the $\ell$th vertex. How this discretization is used to efficiently solve the integral equation, as well as implementation details on the code, are discussed in the forthcoming paper \cite{DirckxPrep}. The elements of the matrix are four-dimensional integrals over triangle pairs. Tensor-Gauss quadrature at orders $3$, $4$, $4$ and $5$ is employed with transformations to remove singularities (see \cite{SauterSchwab2011book} for details).

We consider the sphere and six additional shapes which we refer to as \emph{trefoil knot}, \emph{submarine}, \emph{crankshaft}, \emph{frame}, \emph{falcon} and \emph{lathe part} (see \cref{fig:meshes}). The sphere, trefoil knot and submarine are used at different levels of refinement.

\begin{figure}[tbp]
\input{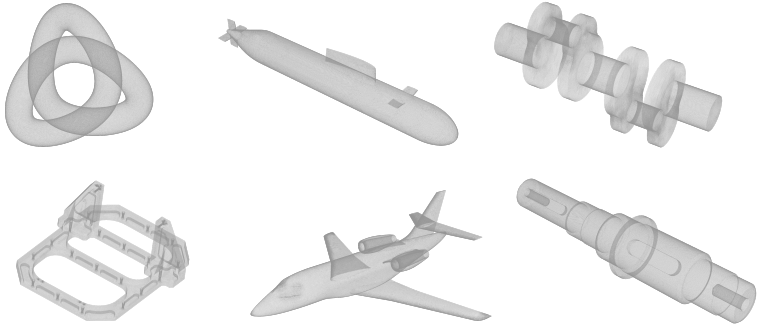}\caption{Shapes used in the numerical experiments. From left to right and from top to bottom: trefoil knot (\SI{548870}{} triangles), submarine (\SI{400886}{} triangles), crankshaft (\SI{672128}{} triangles), frame (\SI{451664}{} triangles), falcon (\SI{609104}{} triangles) and lathe part (\SI{535054}{} triangles).}\label{fig:meshes}
\end{figure}

For the cluster tree $\tree_I=\tree_J$, PCA is employed where clusters are split according to cardinality. A minimal leaf size of $30$ is chosen. Block cluster tree $\tree_{\ItI}$ is built with subdivision strategy $s(\ts):=s(\ct)\times s(\cs)$. Admissibility is based on the weaker geometric admissibility criterion \cref{eq:weak_adm} using approximate diameters and distances based on axis-aligned bounding boxes. As $g$ is symmetric, so is the matrix after discretization; thus only admissible blocks below the diagonal are considered and each $\ct\in\tree_I$ serves both as row and column cluster in accordance with \cref{rmk:symmetry}.

For construction, block-wise approximation is achieved using ACA with \emph{rook pivoting} (ACA\symrook, see \cite[\S3.2]{Dirckx2024} for more details). Relative tolerances are used for both ACA and the uniform compression.\footnote{Note that ACA relies on a heuristic for its error, while in the SVD of the uniform compression the exact rank can be chosen, given a specified relative tolerance on the spectral norm error.} Given a specified tolerance $\epsilon$, the $\Hm$-matrix is constructed using $\epsilon$ as relative tolerance in ACA. To obtain (at least) similar accuracy, the $\UHm$-matrix uses a tolerance of $\epsilon/3$ for ACA and uniform compression. Algebraic recompression of the blocks is performed with $\epsilon/10$ in both cases. Unless specified otherwise, $\epsilon$ is set to $10^{-4}$ in the experiments.

\begin{remark}[Relative global error]\label{rmk:rel-global-err}
    If one expects matrix $A$ to be ``norm-\newline balanced" in the sense that any submatrix $\restr{A}{\ts}$ has spectral norm ${\|\restr{A}{\ts}\|}_2 \sim {\|A\|}_2\sqrt{{|\ct|\cdot|\cs|}/{|I|\cdot|J|}}$, then \cref{thm:spectral_error} and the result in \cite[Lemma 6.26]{Hackbusch2015book} show that using a relative local tolerance $\epsilon$ results in a relative global error ${\|A-A^{\UHm}\|}_2\lesssim {\|A\|}_2\epsilon$.
\end{remark}

All experiments are performed on a compute node consisting of 2 Intel Xeon Platinum 8360Y (Ice Lake) CPUs with 36 cores each, one L3 cache per CPU and 256 GiB total RAM.\footnote{The compute node is part of the `wICE' cluster provided by the VSC (Flemish Supercomputer Center), funded by the Research Foundation -- Flanders (FWO) and the Flemish Government.} Parallel construction and matrix-vector multiplication is achieved according to \cref{ssec:4.4_parallel}. Construction timings are performed 15 times while matrix-vector multiplication timings are run 500 times using three instances of the $\Hm$/$\UHm$-matrix, for a total of 1500 runs. If not specified, the mean time is reported.

\subsection{Sharpness of the admissibility criterion}
The sparsity constant $\csp$ has a significant impact on the attainable gains in $\UHm$-matrix compression. While $\csp$ is bounded independently of the matrix size, this bound is affected by several parameters. The bound based on geometrical observations in \cite{Bebendorf2008book} depends on the geometry of the boundary and the local basis functions defined over it as well as on $\eta$ from the admissibility criterion \cref{eq:strong_adm} or \cref{eq:weak_adm}. Intuitively, a larger value of $\eta$ results in larger but fewer blocks being identified as admissible, thus decreasing $\csp$.

\Cref{fig:numexp-eta} shows the memory usage of both the $\Hm$-matrix and $\UHm$-matrix as a function of $\eta$ for the trefoil knot and the crankshaft at wavenumbers defined relative to the largest triangle side-length $h$ in each shape.\footnote{The values of $\eta$ in \cref{fig:numexp-eta} are chosen relatively large. This compensates for the overestimation of the proximity of clustering using bounding-box-based distances in the admissibility criterion.}

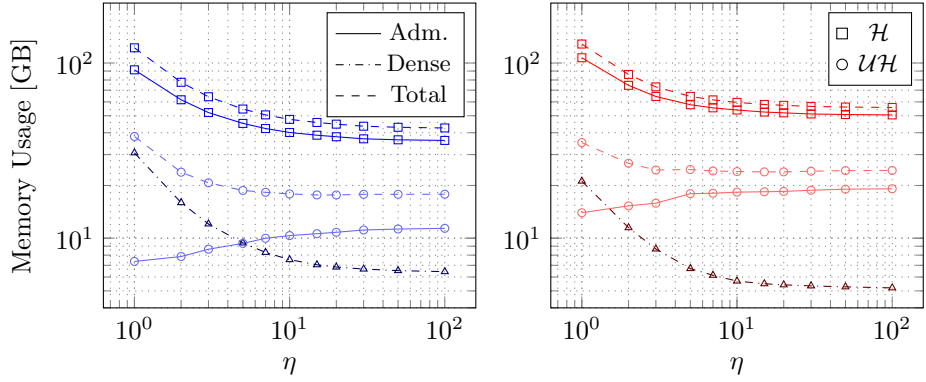
\begin{figure}[tbp]
\subfloat{
    \begin{tikzpicture}[scale=1]
    \pgfplotstableread{M166520-data/fig2/trefoil/memory_mean_H.dat}{\Hdata} 
    \pgfplotstableread{M166520-data/fig2/trefoil/memory_mean_U.dat}{\Udata}     
    \begin{loglogaxis}[width=0.5\linewidth, 
                    xlabel={$\eta$}, 
                    ylabel={Memory Usage [GB]},
                    ymin = 4, ymax = 2.2e2,
                    grid=both,
                    grid style = {dotted,gray},
                    legend pos=north east,
                    legend style = {font=\small}, 
                    legend columns=1]
                    
        \addplot[solid] coordinates {(10,25)};
        \addplot[dashdotted] coordinates {(10,25)};
        \addplot[dashed] coordinates {(10,25)};

        \addplot [blue,mark=square,mark size=1.5pt] table 
            [skip first n=1, x index=0, y expr=\thisrowno{1}/1e9] 
            {\Hdata};
        \addplot [blue!40!black,dashdotted,mark=triangle,mark options=solid,mark size=1.5pt] table 
            [skip first n=1, x index=0, y expr=\thisrowno{2}/1e9] 
            {\Hdata};
        \addplot [blue,dashed,mark=square,mark options=solid,mark size=1.5pt] table 
            [skip first n=1, x index=0, y expr=\thisrowno{3}/1e9] 
            {\Hdata};
        \addplot [blue!60!white,mark=o,mark options=solid,mark size=1.5pt] table 
            [skip first n=1, x index=0, y expr=\thisrowno{1}/1e9] 
            {\Udata};
        \addplot [blue!60!white,dashed,mark=o,mark options=solid,mark size=1.5pt] table 
            [skip first n=1, x index=0, y expr=\thisrowno{3}/1e9] 
            {\Udata};
    \legend{Adm., Dense, Total}
    \end{loglogaxis}
    \end{tikzpicture}
}
\subfloat{
    \begin{tikzpicture}[scale=1]
    \pgfplotstableread{M166520-data/fig2/shaft/memory_mean_H.dat}{\Hdata} 
    \pgfplotstableread{M166520-data/fig2/shaft/memory_mean_U.dat}{\Udata}     
    \begin{loglogaxis}[width=0.5\linewidth,
                    xlabel={$\eta$},
                    ymin = 4, ymax = 2.2e2,
                    grid=both,
                    grid style = {dotted,gray},
                    legend pos=north east,
                    legend style = {font=\small}, 
                    legend columns=1]

        \addplot[only marks,mark=square] coordinates {(100,100)};
        \addplot[only marks,mark=o] coordinates {(100,100)};
        
        \addplot [red,mark=square,mark size=1.5pt] table 
            [skip first n=1, x index=0, y expr=\thisrowno{1}/1e9] 
            {\Hdata};
        \addplot [red!40!black,dashdotted,mark=triangle,mark options=solid,mark size=1.5pt] table 
            [skip first n=1, x index=0, y expr=\thisrowno{2}/1e9] 
            {\Hdata};
        \addplot [red,dashed,mark=square,mark options=solid,mark size=1.5pt] table 
            [skip first n=1, x index=0, y expr=\thisrowno{3}/1e9] 
            {\Hdata};
        \addplot [red!60!white,mark=o,mark options=solid,mark size=1.5pt] table 
            [skip first n=1, x index=0, y expr=\thisrowno{1}/1e9] 
            {\Udata};
        \addplot [red!60!white,dashed,mark=o,mark options=solid,mark size=1.5pt] table 
            [skip first n=1, x index=0, y expr=\thisrowno{3}/1e9] 
            {\Udata};
    \legend{\;$\Hm$, \;$\UHm$}
    \end{loglogaxis}
    \end{tikzpicture}
}
\caption{Memory usage as a function of $\eta$ for the trefoil knot ($\kappa h = 0.1$, left) and crankshaft ($\kappa h = 0.4$, right) in regular and uniform $\Hm$-matrix format.}\label{fig:numexp-eta}
\end{figure}

In both cases, the total memory usage decreases to an asymptote with increasing $\eta$. If $\eta \to \infty$, this could be understood as considering a block admissible if its clusters' bounding boxes do not overlap. In the case of the $\Hm$-matrix, both the admissible and dense storage benefit from relaxing the admissibility criterion. The dense storage reduces as blocks turn admissible while the admissible storage reduces as children blocks are aggregated into a parent block with less memory usage. The aggregation directly results in a lower sparsity constant. The $\UHm$-matrix format is affected by this to the point where increase of $\eta$ leads to increase of the admissible storage. We observe that the $\UHm$-matrix is less affected by a poor choice of $\eta$, which can be seen as a qualitative advantage as it becomes easier to choose a `good' value of $\eta$.

Based on \cref{fig:numexp-eta}, $\eta=10$ is chosen for the remainder of the experiments.

\subsection{Asymptotics in the problem size}
To verify the log-linear complexity of the compression scheme, it is applied to three sets of iteratively refined meshes. \Cref{fig:numexp-complexity} shows the memory usage and construction time for the trefoil knot over a range of refinements. This is done for the Laplace kernel, obtained by setting $\kappa=0$, as well as for the Helmholtz kernel with $\kappa h=0.3$. Both the sequential and parallel construction times are reported, except for the largest meshes.

\begin{figure}[tbp]
\centering
\subfloat{
    \begin{tikzpicture}[scale=1]
    \pgfplotstableread{M166520-data/fig3/trefoil-laplace/memperdof_mean_H.dat}{\LaplaceHdata}
    \pgfplotstableread{M166520-data/fig3/trefoil-laplace/memperdof_mean_U.dat}{\LaplaceUdata}
    \pgfplotstableread{M166520-data/fig3/trefoil-helmholtz/memperdof_mean_H.dat}{\HelmholtzHdata}
    \pgfplotstableread{M166520-data/fig3/trefoil-helmholtz/memperdof_mean_U.dat}{\HelmholtzUdata}
    \begin{semilogxaxis}[width=0.47\linewidth, 
                        xlabel={$N$}, 
                        ylabel={Memory/DOF [KB]}, 
                        ymin=0,
                        grid=major, 
                        grid style = {dotted,gray},
                        legend pos=north west, 
                        legend style = {font=\small}, 
                        legend columns=1]
        
        \addplot[mark=square,mark size=1.5pt] coordinates {(1e4,41)};
        \addplot[dashed,mark=o,mark options=solid,mark size=1.5pt] coordinates {(1e4,41)};
        
        \addplot [blue,mark=square,mark size=1.5pt] table 
            [skip first n=1, x expr=\thisrowno{0}*3, y expr=\thisrowno{1}/1e3] 
            {\LaplaceHdata};
        \addplot [blue,dashed,mark=o,mark options=solid,mark size=1.5pt] table 
            [skip first n=1, x expr=\thisrowno{0}*3, y expr=\thisrowno{1}/1e3] 
            {\LaplaceUdata};
        \addplot [orange,mark=square,mark size=1.5pt] table 
            [skip first n=1, x expr=\thisrowno{0}*3, y expr=\thisrowno{1}/1e3] 
            {\HelmholtzHdata};
        \addplot [orange,dashed,mark=o,mark options=solid,mark size=1.5pt] table 
            [skip first n=1, x expr=\thisrowno{0}*3, y expr=\thisrowno{1}/1e3] 
            {\HelmholtzUdata};
    \legend{$\Hm$,$\UHm$}
    \end{semilogxaxis}
    \end{tikzpicture}
}
\subfloat{
    \begin{tikzpicture}[scale=1]
    \pgfplotstableread{M166520-data/fig3/trefoil-laplace/construction-time_mean_H.dat}{\LaplaceMeanHdata}
    \pgfplotstableread{M166520-data/fig3/trefoil-laplace/construction-time_mean_U.dat}{\LaplaceMeanUdata}
    \pgfplotstableread{M166520-data/fig3/trefoil-helmholtz/construction-time_mean_H.dat}{\HelmholtzMeanHdata}
    \pgfplotstableread{M166520-data/fig3/trefoil-helmholtz/construction-time_mean_U.dat}{\HelmholtzMeanUdata}

    \pgfplotstableread{M166520-data/fig3/trefoil-laplace-seq/construction-time_mean_H.dat}{\LaplaceSeqMeanHdata}
    \pgfplotstableread{M166520-data/fig3/trefoil-laplace-seq/construction-time_mean_U.dat}{\LaplaceSeqMeanUdata}
    \pgfplotstableread{M166520-data/fig3/trefoil-helmholtz-seq/construction-time_mean_H.dat}{\HelmholtzSeqMeanHdata}
    \pgfplotstableread{M166520-data/fig3/trefoil-helmholtz-seq/construction-time_mean_U.dat}{\HelmholtzSeqMeanUdata}
    
    \begin{loglogaxis}[width=0.47\linewidth, 
                    xlabel={$N$}, 
                    ylabel={Construction time [s]}, 
                    grid=major, 
                    grid style = {dotted,gray},
                    legend pos=north west, 
                    legend style = {font=\small}, 
                    legend columns=1]

        \addplot [blue,mark=square,mark size=1.5pt] table 
            [skip first n=1, x expr=\thisrowno{0}*3, y index=2] {\LaplaceMeanHdata};
        \addplot [blue,dashed,mark=o,mark options=solid,mark size=1.5pt] table 
            [skip first n=1, x expr=\thisrowno{0}*3, y index=2] {\LaplaceMeanUdata};
        \addplot [orange,mark=square,mark size=1.5pt] table 
            [skip first n=1, x expr=\thisrowno{0}*3, y index=2] {\HelmholtzMeanHdata};
        \addplot [orange,dashed,mark=o,mark options=solid,mark size=1.5pt] table 
            [skip first n=1, x expr=\thisrowno{0}*3, y index=2] {\HelmholtzMeanUdata}
            node[above,pos=1,xshift=-10,yshift=-40] {\color{black}\footnotesize$p=72$};

        \addplot [blue,mark=square,mark size=1.5pt] table 
            [skip first n=1, x expr=\thisrowno{0}*3, y index=2] {\LaplaceSeqMeanHdata};
        \addplot [blue,dashed,mark=o,mark options=solid,mark size=1.5pt] table 
            [skip first n=1, x expr=\thisrowno{0}*3, y index=2] {\LaplaceSeqMeanUdata};
        \addplot [orange,mark=square,mark size=1.5pt] table 
            [skip first n=1, x expr=\thisrowno{0}*3, y index=2] {\HelmholtzSeqMeanHdata};
        \addplot [orange,dashed,mark=o,mark options=solid,mark size=1.5pt] table 
            [skip first n=1, x expr=\thisrowno{0}*3, y index=2] {\HelmholtzSeqMeanUdata}
            node[above,pos=1,xshift=0,yshift=0] {\color{black}\footnotesize$p=1$};

        \addplot [black,dashed,domain=9e3:3e6,samples=100] {2e-2*(x/3)*ln(x/3)}
            node[above,pos=1,xshift=-55,yshift=-10] {\footnotesize$\bigO(N\log N)$};
    \end{loglogaxis}
    \end{tikzpicture}
}
\caption{Memory usage and construction time as a function of $N$ for the trefoil knot in regular and uniform $\Hm$-matrix format at $\kappa h=0$ (blue) and $\kappa h=0.3$ (orange).}\label{fig:numexp-complexity}
\end{figure}
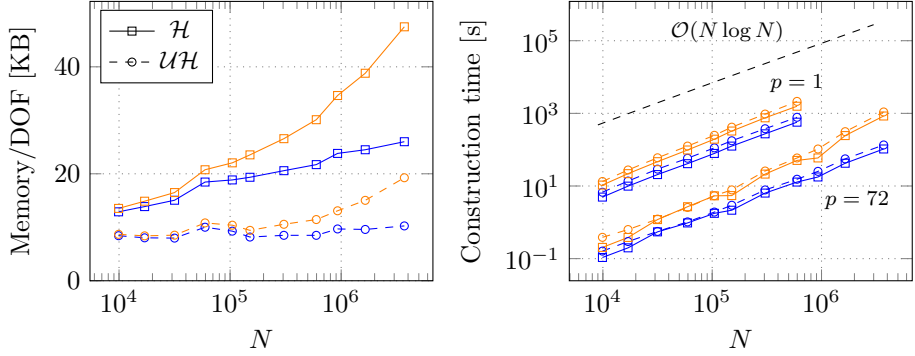

The memory usage of both the $\Hm$-matrix and $\UHm$-matrix clearly has $\bigO(N\log N)$ complexity in the case of the Laplace kernel. Given that $\kappa$ increases with $N$ if $\kappa h \sim 1$, it is not expected that the Helmholtz kernel scales as $\bigO(N\log N)$; however it remains close here. The construction time scales similarly except for the parallel construction at the largest $N$, where additional memory-architecture effects play a role. Noteworthy is the observation that the $\UHm$-matrix construction not only scales with $\bigO(N\log N)$ as expected but also parallelizes as well as the $\Hm$-matrix construction.

\Cref{tab:numexp-complexity} compiles the results of the sphere and the submarine for $\kappa=0$. Similar observations can be made w.r.t.\ the complexity as in \cref{fig:numexp-complexity}. The $\UHm$-matrix is able to attain compression factors of over 2 while the construction time is comparable. \Cref{tab:numexp-complexity} can be compared to \cite[Table 4.2]{Borm2009a} where $\Hm$-matrices are compressed to the $\Hm^2$-matrix format. For larger $N$, the linear storage complexity of $\Hm^2$-matrices results in larger compression ratios than reported here. However, the total construction time still scales log-linearly. Due to the recursive $\Hm^2$-matrix structure, the compression scheme is also much more involved and requires more implementation effort.
\begin{table}[thbp]
	\caption{Memory usage per degree of freedom in KB and construction time (sequential and parallel) in seconds for the increasingly refined sphere (top) and submarine (bottom) at $\kappa=0$.}\label{tab:numexp-complexity}
	\footnotesize
    \centering
    \begin{filecontents*}{M166520-data/tab1/sphere-table.csv}
NtP,khHP,memtotHP,memadmHP,memperdofHP,minBuildHP,meanBuildHP,khUP,memtotUP,memadmUP,memperdofUP,minBuildUP,meanBuildUP,NtS,khHS,memtotHS,memadmHS,memperdofHS,minBuildHS,meanBuildHS,khUS,memtotUS,memadmUS,memperdofUS,minBuildUS,meanBuildUS
2048.0,0.0,53537280.0,29465088.0,8713.75,0.033792892,0.0360635285333333,0.0,34152666.666666664,10080474.666666666,5558.702256944444,0.046682471,0.0493587144666666,2048.0,0.0,53510144.0,29437952.0,8709.333333333334,2.080336394,2.1046168048000005,0.0,34159834.666666664,10087642.666666666,5559.868923611112,2.690161817,2.7282172038
8192.0,0.0,272613888.0,173965824.0,11092.6875,0.165824677,0.1723589963333333,0.0,159365264.0,60717200.0,6484.589192708333,0.218391731,0.2217685929333333,8192.0,0.0,272701440.0,174053376.0,11096.25,10.96699915,11.064451722933336,0.0,159371776.0,60723712.0,6484.854166666667,13.830792998,13.9825670962
32768.0,0.0,1321348096.0,936487936.0,13441.447916666666,0.836784554,0.9084270266,0.0,672693872.0,287833712.0,6842.995930989583,1.108186329,1.1370763656666667,32768.0,0.0,1322797056.0,937936896.0,13456.1875,54.860519725,55.455569478,0.0,672596608.0,287736448.0,6842.006510416667,69.785663776,70.3040894
131072.0,0.0,6189586432.0,4631934976.0,15740.932291666666,4.356082234,4.7403328816,0.0,2936914992.0,1379263536.0,7468.961059570312,5.249009525,5.6172434136666665,131072.0,0.0,6188356608.0,4630705152.0,15737.8046875,264.679282972,269.59329435959995,0.0,2936894026.6666665,1379242570.6666667,7468.907741970487,334.577250815,337.1242298303333
524288.0,0.0,28910627328.0,22620781056.0,18380.8818359375,22.631856805,24.455094908333336,0.0,12828891594.666666,6539045322.666667,8156.389614529081,28.273081563,29.050254145466667,{},{},{},{},{},{},{},{},{},{},{},{},{}
2097152.0,0.0,131513022976.0,106390391296.0,20903.432047526047,122.671599335,132.99140259733335,0.0,54366490933.333336,29243859253.33333,8641.321012708877,157.119874111,168.43346286686665,{},{},{},{},{},{},{},{},{},{},{},{},{}
\end{filecontents*}
\csvreader[
        head to column names,
        tabular = ccccccc,
        table head = \toprule & \multicolumn{3}{c}{$\mathcal{H}$-matrix} & \multicolumn{3}{c}{$\mathcal{UH}$-matrix} \\ \cmidrule(r){2-4}     \cmidrule(r){5-7} {$N$} & {$\mathrm{Mem}/N$} & {Build}(1) & {Build}(72) & {$\mathrm{Mem}/N$} & {Build}(1) & {Build}(72) \\ \midrule,
        table foot = \midrule,
    ]{M166520-data/tab1/sphere-table.csv}{}{
        \tablenum[round-precision=0, round-mode=places, table-format=7.0]{\fpeval{3*\NtP}} & 
        \tablenum[round-precision=2, round-mode=places, table-format=2.2]{\fpeval{\memperdofHP/1e3}} & 
        \tablenum[round-precision=2, round-mode=places, exponent-mode=scientific, table-format=1.2e1]{\meanBuildHS} & 
        \tablenum[round-precision=2, round-mode=places, exponent-mode=scientific, table-format=1.2e1]{\meanBuildHP} & 
        \tablenum[round-precision=2, round-mode=places, table-format=2.2]{\fpeval{\memperdofUP/1e3}} & 
        \tablenum[round-precision=2, round-mode=places, exponent-mode=scientific, table-format=1.2e1]{\meanBuildUS} &
        \tablenum[round-precision=2, round-mode=places, exponent-mode=scientific, table-format=1.2e1]{\meanBuildUP}
    }
    \begin{filecontents*}{M166520-data/tab1/submarine-table.csv}
NtP,khHP,memtotHP,memadmHP,memperdofHP,minBuildHP,meanBuildHP,khUP,memtotUP,memadmUP,memperdofUP,minBuildUP,meanBuildUP,NtS,khHS,memtotHS,memadmHS,memperdofHS,minBuildHS,meanBuildHS,khUS,memtotUS,memadmUS,memperdofUS,minBuildUS,meanBuildUS
4140.0,0.0,116426752.0,55291264.0,9374.1346215781,0.089950609,0.0979179314,0.0,88607050.66666667,27471562.666666668,7134.223081052067,0.085300658,0.0918166184666666,4140.0,0.0,116572800.0,55437312.0,9385.893719806763,3.964591158,4.019085874333333,0.0,88652570.66666667,27517082.666666668,7137.888137412775,5.057149273,5.109611387666666
10406.0,0.0,312551088.0,224001024.0,10011.886988275995,0.303404672,0.3290506488,0.0,189220576.0,100478992.0,6061.265167531553,0.32401424,0.3700146100666666,10406.0,0.0,312407952.0,223872768.0,10007.301941187774,13.272454027,13.416233207066666,0.0,189270725.3333333,100514741.33333331,6062.87159117603,17.950534962,18.09347186686667
31104.0,0.0,1164025456.0,788638624.0,12474.552640603564,0.822658541,0.92383012,0.0,688935253.3333334,313263253.3333333,7383.136716963877,1.010667655,1.0530888001333334,31104.0,0.0,1165066784.0,789636752.0,12485.712277091909,46.229922594,46.71897983120001,0.0,688901344.0,313142224.0,7382.773319615912,61.12390877,61.6554192618
106888.0,0.0,4847672832.0,3729749664.0,15117.608562233369,3.855020381,4.292486776066666,0.0,2401838213.333333,1283043317.3333333,7490.202247004134,4.669303128,5.0566692568,106888.0,0.0,4838629984.0,3720715552.0,15089.408178030586,211.123910985,213.22390333240003,0.0,2402958362.6666665,1284220874.6666667,7493.695465242953,277.040645132,279.1566699324667
400886.0,0.0,20956136976.0,17116864848.0,17424.851434073527,17.6739507,19.379983062466668,0.0,9364414906.666666,5523803290.666667,7786.432141695034,22.756840078,24.00860278453333,{},{},{},{},{},{},{},{},{},{},{},{},{}
1584014.0,0.0,94160201072.0,79232132336.0,19814.681998180997,88.100339294,98.83699799873334,0.0,39478862021.333336,24549418373.33333,8307.767907214064,120.25095173,128.46530622753335,{},{},{},{},{},{},{},{},{},{},{},{},{}
\end{filecontents*}
\csvreader[
        head to column names,
        tabular = ccccccc,
        table head = \hphantom{$N$} & \hphantom{$\mathrm{Mem}/N$} & \hphantom{Build} & \hphantom{Build} & \hphantom{$\mathrm{Mem}/N$} & \hphantom{Build} & \hphantom{Build} \vspace{-3ex}\\,
        table foot = \bottomrule,
    ]{M166520-data/tab1/submarine-table.csv}{}{
        \tablenum[round-precision=0, round-mode=places, table-format=7.0]{\fpeval{3*\NtP}} & 
        \tablenum[round-precision=2, round-mode=places, table-format=2.2]{\fpeval{\memperdofHP/1e3}} & 
        \tablenum[round-precision=2, round-mode=places, exponent-mode=scientific, table-format=1.2e1]{\meanBuildHS} &    
        \tablenum[round-precision=2, round-mode=places, exponent-mode=scientific, table-format=1.2e1]{\meanBuildHP} & 
        \tablenum[round-precision=2, round-mode=places, table-format=2.2]{\fpeval{\memperdofUP/1e3}} & 
        \tablenum[round-precision=2, round-mode=places, exponent-mode=scientific, table-format=1.2e1]{\meanBuildUS} &
        \tablenum[round-precision=2, round-mode=places, exponent-mode=scientific, table-format=1.2e1]{\meanBuildUP}
    }
\end{table}

\subsection{Behavior with respect to tolerance}

The storage cost and construction time of $\Hm$- and $\UHm$-matrices implicitly depend on the tolerance through the block ranks $(k_b)_{b}$ and cluster ranks $(\ell_\ct)_{\ct}$ and $(\ell_\cs)_{\cs}$. \Cref{fig:numexp-tolerance} shows the memory usage and construction time as a function of the relative tolerance $\epsilon$ for the crankshaft and submarine. Default quadrature orders of $(3,4,4,5)$ were used at precisions $10^{-3}$ and $10^{-4}$, increased to $(4,5,5,6)$ for $\epsilon=10^{-5}$ and decreased to $(2,3,3,4)$ for $\epsilon=10^{-2}$.

\begin{figure}[thbp]
\small
\begin{tikzpicture}
\begin{groupplot}[group style={group size=3 by 2, horizontal sep=0.06\linewidth, vertical sep=0.05\linewidth}]

    \pgfplotstableread{M166520-data/fig4/shaft/memory_mean_H.dat}{\Hdata} 
    \pgfplotstableread{M166520-data/fig4/shaft/memory_mean_U.dat}{\Udata}    

    \pgfplotstableread{M166520-data/fig4/submarine/memory_mean_H.dat}{\betssiHdata} 
    \pgfplotstableread{M166520-data/fig4/submarine/memory_mean_U.dat}{\betssiUdata}     

    \pgfplotstableread{M166520-data/fig4/shaft/memory_ratio.dat}{\Rdata} 
    \pgfplotstableread{M166520-data/fig4/submarine/memory_ratio.dat}{\betssiRdata} 

    \pgfplotstableread{M166520-data/fig4/shaft/construction-time_min_H.dat}{\minHdata} 
    \pgfplotstableread{M166520-data/fig4/shaft/construction-time_min_U.dat}{\minUdata}     
    \pgfplotstableread{M166520-data/fig4/shaft/construction-time_mean_H.dat}{\meanHdata} 
    \pgfplotstableread{M166520-data/fig4/shaft/construction-time_mean_U.dat}{\meanUdata}  

     \pgfplotstableread{M166520-data/fig4/submarine/construction-time_min_H.dat}{\betssiMinHdata} 
    \pgfplotstableread{M166520-data/fig4/submarine/construction-time_min_U.dat}{\betssiMinUdata}     
    \pgfplotstableread{M166520-data/fig4/submarine/construction-time_mean_H.dat}{\betssiMeanHdata} 
    \pgfplotstableread{M166520-data/fig4/submarine/construction-time_mean_U.dat}{\betssiMeanUdata}  

    \pgfplotstableread{M166520-data/fig4/shaft/construction-time_ratio.dat}{\timeRdata} 
    \pgfplotstableread{M166520-data/fig4/submarine/construction-time_ratio.dat}{\timeBetssiRdata} 

    \nextgroupplot[width=0.25\linewidth, 
                    scale only axis,
                    xmode=log,
                    ylabel={Memory Usage [GB]}, 
                    ymin=0, ymax=72,
                    grid=major,
                    grid style = {dotted,gray},
                    legend pos=north east, 
                    legend style = {font=\scriptsize}]

        \addplot[solid] coordinates {(1e-4,25)};
        \addplot[dashdotted] coordinates {(1e-4,25)};
        \addplot[dashed] coordinates {(1e-4,25)};

        \addplot [blue,mark=square,mark size=1.5pt] table 
            [skip first n=1, x index=0, y expr=\thisrowno{1}/1e9] 
            {\Hdata};
        \addplot [blue!50!black,dashdotted,mark size=1.5pt] table 
            [skip first n=1, x index=0, y expr=\thisrowno{2}/1e9] 
            {\Hdata};
        \addplot [blue,dashed,mark=square,mark options=solid,mark size=1.5pt] table 
            [skip first n=1, x index=0, y expr=\thisrowno{3}/1e9] 
            {\Hdata};
        \addplot [blue!50!white,mark=o,mark options=solid,mark size=1.5pt] table 
            [skip first n=1, x index=0, y expr=\thisrowno{1}/1e9] 
            {\Udata};
        \addplot [blue!50!white,dashed,mark=o,mark options=solid,mark size=1.5pt] table 
            [skip first n=1, x index=0, y expr=\thisrowno{3}/1e9] 
            {\Udata};
    \legend{Adm., Dense, Total}

    \nextgroupplot[width=0.25\linewidth, 
                    scale only axis,
                    xmode=log,
                    ymin=0, ymax=35,
                    grid=major,
                    grid style={dotted,gray},
                    legend pos=north east, 
                    legend style = {font=\scriptsize}]

        \addplot[only marks,mark=square] coordinates {(100,100)};
        \addplot[only marks,mark=o] coordinates {(100,100)};
        
        \addplot [red,mark=square,mark size=1.5pt] table 
            [skip first n=1, x index=0, y expr=\thisrowno{1}/1e9] 
            {\betssiHdata};
        \addplot [red!50!black,dashdotted,mark size=1.5pt] table 
            [skip first n=1, x index=0, y expr=\thisrowno{2}/1e9] 
            {\betssiHdata};
        \addplot [red,dashed,mark=square,mark options=solid,mark size=1.5pt] table 
            [skip first n=1, x index=0, y expr=\thisrowno{3}/1e9] 
            {\betssiHdata};
        \addplot [red!50!white,mark=o,mark options=solid,mark size=1.5pt] table 
            [skip first n=1, x index=0, y expr=\thisrowno{1}/1e9] 
            {\betssiUdata};
        \addplot [red!50!white,dashed,mark=o,mark options=solid,mark size=1.5pt] table 
            [skip first n=1, x index=0, y expr=\thisrowno{3}/1e9] 
            {\betssiUdata};
    \legend{\;$\Hm$, \;$\UHm$}

    \nextgroupplot[width=0.25\linewidth, 
                    scale only axis,
                    xmode=log,
                    ymin = 1.7, ymax = 5.9,
                    grid=major, 
                    grid style={dotted,gray},
                    legend pos=north east, 
                    legend style = {font=\scriptsize}]
        \addplot [blue,mark=triangle] table 
            [skip first n=1, x index=0, y index=1] {\Rdata};
        \addplot [red,mark=x, mark size=2.5] table 
            [skip first n=1, x index=0, y index=1] {\betssiRdata};
        \addplot [blue,dashed,mark=triangle,mark options=solid] table 
            [skip first n=1, x index=0, y index=3] {\Rdata};
        \addplot [red,dashed,mark=x,mark options=solid, mark size=2.5] table 
            [skip first n=1, x index=0, y index=3] {\betssiRdata};
    \legend{\emph{crankshaft}, \emph{submarine}}

    \nextgroupplot[width=0.25\linewidth, 
                    scale only axis,
                    xmode=log, ymode=log,
                    ymin = 8, ymax = 650,
                    xlabel={Tolerance $\epsilon$}, 
                    ylabel={Construction Time [s]},
                    grid=major,
                    grid style = {dotted,gray},
                    legend pos=north east, 
                    legend style = {font=\scriptsize}]
                    
        \addplot[solid] coordinates {(1e-4,500)};
        \addplot[dashed] coordinates {(1e-4,500)};
                
        \addplot [blue,mark=square,mark size=1.5pt] table [skip first n=1, x index=0, y index=2] {\meanHdata};
        \addplot [blue,dashed,mark=square,mark options=solid,mark size=1.5pt] table [skip first n=1, x index=0, y index=1] {\minHdata};
        \addplot [blue!50!white,mark=o,mark size=1.5pt] table [skip first n=1, x index=0, y index=2] {\meanUdata};
        \addplot [blue!50!white,dashed,mark=o,mark options=solid,mark size=1.5pt] table [skip first n=1, x index=0, y index=1] {\minUdata};
    \legend{Mean,Minimum}

    \nextgroupplot[width=0.25\linewidth, 
                    scale only axis,
                    xmode=log, ymode=log,
                    ymin = 5, ymax = 350,
                    xlabel={Tolerance $\epsilon$}, 
                    grid=major, 
                    grid style={dotted,gray}]
        \addplot [red,mark=square,mark size=1.5pt] table [skip first n=1, x index=0, y index=2] {\betssiMeanHdata};
        \addplot [red,dashed,mark=square,mark options=solid,mark size=1.5pt] table [skip first n=1, x index=0, y index=1] {\betssiMinHdata};
        \addplot [red!50!white,mark=o,mark size=1.5pt] table [skip first n=1, x index=0, y index=2] {\betssiMeanUdata};
        \addplot [red!50!white,dashed,mark=o,mark options=solid,mark size=1.5pt] table [skip first n=1, x index=0, y index=1] {\betssiMinUdata};

    \nextgroupplot[width=0.25\linewidth, 
                    scale only axis,
                    xmode=log,
                    xlabel={Tolerance $\epsilon$}, 
                    ymin=0.74, ymax=0.86,
                    grid=major, 
                    grid style={dotted,gray},
                    legend pos=north east, 
                    legend style = {font=\small}, 
                    legend columns=2]
        \addplot [blue,mark=triangle] table [skip first n=1, x index=0, y index=2] {\timeRdata};
        \addplot [blue,dashed,mark=triangle,mark options=solid] table [skip first n=1, x index=0, y index=1] {\timeRdata};
        \addplot [red,mark=x, mark size=2.5] table [skip first n=1, x index=0, y index=2] {\timeBetssiRdata};
        \addplot [red,dashed,mark=x,mark options=solid, mark size=2.5] table [skip first n=1, x index=0, y index=1] {\timeBetssiRdata};

\end{groupplot}
\end{tikzpicture}
\caption{Memory usage and parallel construction time as a function of the tolerance $\epsilon$ for the crankshaft and submarine at $\kappa h=0.1$ in regular and uniform $\Hm$-matrix format. Left: Absolute values for the crankshaft. Middle: Absolute values for the submarine. Right: Relative values $\Hm/\UHm$.}\label{fig:numexp-tolerance}
\end{figure}
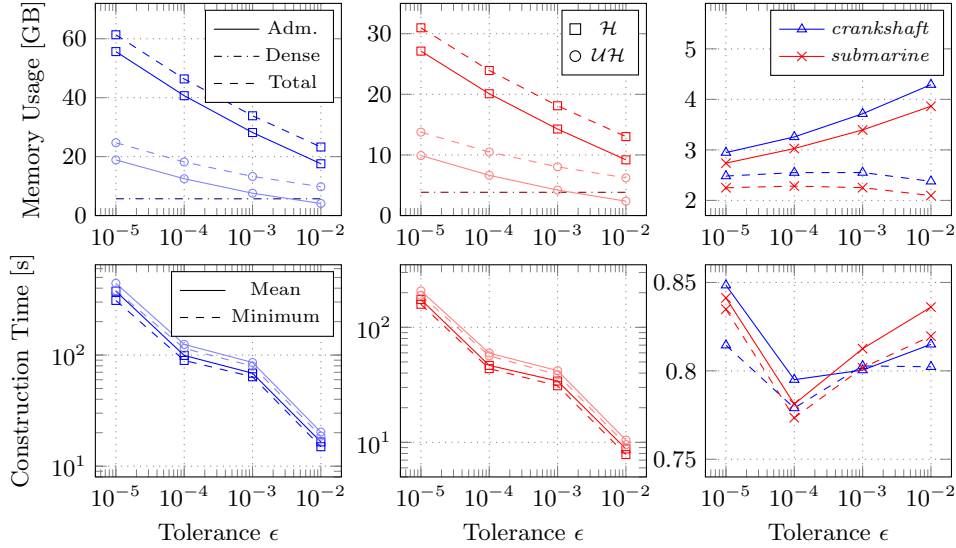

Memory usage of the $\Hm$- and $\UHm$-matrices behaves as $\bigO(\log(1/\epsilon))$, to be expected as asymptotic smoothness of the kernel results in exponentially decaying singular values of admissible blocks. Total compression factors are stable throughout the $\epsilon$-range.

The expected behavior of construcion time w.r.t.\ the tolerance is more complicated as different parts of the construction have rank dependencies of $k$, $k^2$ and $k^3$. Additionally, as the sampling of the matrix elements has a $\bigO(q^4)$ dependence on the quadrature order $q$, this affects the construction time as well, which is very clear at $\epsilon=10^{-2},10^{-5}$. For both shapes, the $\UHm$-matrix construction takes slightly longer, adding an additional 25\% to the runtime on average.

The block- and cluster-wise relative tolerances in the experimental set-up do not guarantee a relative tolerance on the whole matrix. Therefore, \cref{fig:numexp-accuracy} plots the total memory usage against the global relative spectral error, estimated using power iteration, for the crankshaft and submarine. The error is averaged over 15 constructed matrices. The same values for $\epsilon$ are used as in \cref{fig:numexp-tolerance}.

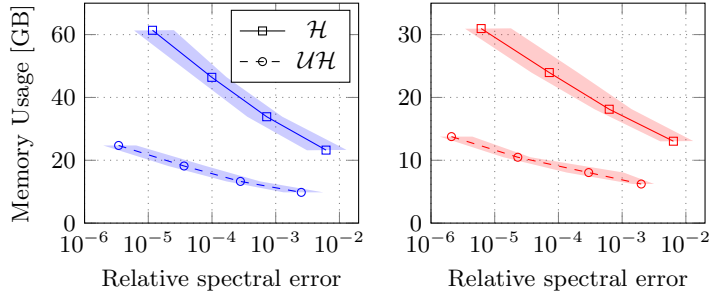
\begin{figure}[thbp]
\small
\centering
\subfloat{
\begin{tikzpicture}[scale=1]
    \pgfplotstableread{M166520-data/fig5/shaft/tol-mem-error_mean_H.dat}{\shaftHdata}
    \pgfplotstableread{M166520-data/fig5/shaft/tol-mem-error_mean_U.dat}{\shaftUdata}
    \pgfplotstableread{M166520-data/fig5/shaft/tol-mem-error_min_H.dat}{\shaftHdataMin}
    \pgfplotstableread{M166520-data/fig5/shaft/tol-mem-error_min_U.dat}{\shaftUdataMin}
    \pgfplotstableread{M166520-data/fig5/shaft/tol-mem-error_max_H.dat}{\shaftHdataMax}
    \pgfplotstableread{M166520-data/fig5/shaft/tol-mem-error_max_U.dat}{\shaftUdataMax}
    \begin{semilogxaxis}[
            width=0.4\linewidth, 
            xlabel={Relative spectral error},  
            ylabel={Memory Usage [GB]},
            xmin=1e-6, xmax=2e-2,
            ymin=0, ymax=70,
            xtick={1e-6,1e-4,1e-2},
            extra x ticks={1e-5,1e-3},
            grid=major,
            grid style = {dotted,gray},
            legend pos=north east,
            legend style = {font=\small}]
                    
        \addplot[mark=square,mark size=1.5pt] coordinates {(100,100)};
        \addplot[dashed,mark=o,mark options=solid,mark size=1.5pt] coordinates {(100,100)};
        
        \addplot [name path=minh, draw=none] table 
            [skip first n=1, x index=2, y expr=\thisrowno{1}/1e9] 
            {\shaftHdataMin};
        \addplot [blue,mark=square,mark size=1.5pt] table 
            [skip first n=1, x index=3, y expr=\thisrowno{1}/1e9] 
            {\shaftHdata};
        \addplot [name path=maxh, draw=none] table 
            [skip first n=1, x index=4, y expr=\thisrowno{1}/1e9] 
            {\shaftHdataMax};
        \addplot[blue,opacity=0.2] fill between [of=minh and maxh];

        \addplot [name path=minu, draw=none] table 
            [skip first n=1, x index=2, y expr=\thisrowno{1}/1e9] 
            {\shaftUdataMin};
        \addplot [blue,dashed,mark=o,mark options=solid,mark size=1.5pt] table 
            [skip first n=1, x index=3, y expr=\thisrowno{1}/1e9] 
            {\shaftUdata};
        \addplot [name path=maxu, draw=none] table 
            [skip first n=1, x index=4, y expr=\thisrowno{1}/1e9] 
            {\shaftUdataMax};
        \addplot[blue,opacity=0.2] fill between [of=minu and maxu];

    \legend{\;$\Hm$, \;$\UHm$}
    \end{semilogxaxis}
\end{tikzpicture}
}
\subfloat{
\begin{tikzpicture}[scale=1]
    \pgfplotstableread{M166520-data/fig5/submarine/tol-mem-error_mean_H.dat}{\betssiHdata}
    \pgfplotstableread{M166520-data/fig5/submarine/tol-mem-error_mean_U.dat}{\betssiUdata}
    \pgfplotstableread{M166520-data/fig5/submarine/tol-mem-error_min_H.dat}{\betssiHdataMin}
    \pgfplotstableread{M166520-data/fig5/submarine/tol-mem-error_min_U.dat}{\betssiUdataMin}
    \pgfplotstableread{M166520-data/fig5/submarine/tol-mem-error_max_H.dat}{\betssiHdataMax}
    \pgfplotstableread{M166520-data/fig5/submarine/tol-mem-error_max_U.dat}{\betssiUdataMax}
    \begin{semilogxaxis}[
            width=0.4\linewidth,
            xlabel={Relative spectral error},
            xmin=1e-6, xmax=2e-2,
            ymin=0, ymax=35,
            xtick={1e-6,1e-4,1e-2},
            extra x ticks={1e-5,1e-3},
            grid=major,
            grid style = {dotted,gray},
            legend pos=north east,
            legend style = {font=\small}]

        \addplot [name path=minh, draw=none] table 
            [skip first n=1, x index=2, y expr=\thisrowno{1}/1e9] 
            {\betssiHdataMin};
        \addplot [red,mark=square,mark size=1.5pt] table 
            [skip first n=1, x index=3, y expr=\thisrowno{1}/1e9] 
            {\betssiHdata};
        \addplot [name path=maxh, draw=none] table 
            [skip first n=1, x index=4, y expr=\thisrowno{1}/1e9] 
            {\betssiHdataMax};
        \addplot[red,opacity=0.2] fill between [of=minh and maxh];

        \addplot [name path=minu, draw=none] table 
            [skip first n=1, x index=2, y expr=\thisrowno{1}/1e9] 
            {\betssiUdataMin};
        \addplot [red,dashed,mark=o,mark options=solid,mark size=1.5pt] table 
            [skip first n=1, x index=3, y expr=\thisrowno{1}/1e9] 
            {\betssiUdata};
        \addplot [name path=maxu, draw=none] table 
            [skip first n=1, x index=4, y expr=\thisrowno{1}/1e9] 
            {\betssiUdataMax};
        \addplot[red,opacity=0.2] fill between [of=minu and maxu];
    \end{semilogxaxis}
\end{tikzpicture}
}
\caption{Total memory usage as a function of the relative spectral error $\|A-A^\bullet\|_2/\|A\|_2$ ($\bullet=\Hm,\UHm$) for the crankshaft (left) and submarine (right) at $\kappa h=0.1$. Lines indicate the mean error while the light bands indicate the minimum and maximum errors.}\label{fig:numexp-accuracy}
\end{figure}

The spectral errors behave mostly as desired, only the $\Hm$-matrices exceed the imposed tolerance $\epsilon$ at times. Using the stricter tolerance $\epsilon/3$, the $\UHm$-matrices consistently achieve lower errors than the $\Hm$-matrices at tolerance $\epsilon$. It shows that the uniform compression may even be more effective than the current set-up suggests.

Besides this experiment, we have also estimated the relative spectral errors across all other numerical tests. Of all the computed errors, none exceeded the imposed $\epsilon$ by more than a factor of $10$. In addition, the $\UHm$-matrices were more accurate than their $\Hm$-matrix counterparts in 97.9\% of cases.

\subsection{Memory reduction in the $\UHm$-matrix format}
Finally, robustness of the uniform compression is tested by applying it to a total of six shapes with varying values for $\kappa h$. The results are collected in \cref{tab:numexp-meshStats}. The $\Hm$-matrix results are given in absolute values while the $\UHm$-matrix results are given in relative terms.
\begin{table}[thbp]
	\caption{Memory usage in GB, and parallel construction time and matrix-vector product timings (MV) in seconds for the $\Hm$- and $\UHm$-matrix of six different shapes.}\label{tab:numexp-meshStats}
	\footnotesize
    \centering
    \setlength{\tabcolsep}{5pt} 
    \begin{filecontents*}{M166520-data/tab2/compression_table.csv}
shape,khH,NtH,memtotH,memadmH,minBuildH,meanBuildH,minMatvecH,meanMatvecH,khU,NtU,memtotU,memadmU,minBuildU,meanBuildU,minMatvecU,meanMatvecU,
trefoil knot,0.09182080399999999,548870.0,47040282336.0,39497995056.0,100.200152316,108.33431064366668,0.401672803,0.4660813063546667,0.09182080399999999,548870.0,17648591274.666668,10106303994.666666,122.500058376,136.38440045133333,0.185311201,0.20011487014066665,
submarine,0.27235723,400886.0,29583025040.0,25743737888.0,64.920420292,69.97880433553334,0.256249768,0.29253408730400005,0.27235723,400886.0,13143895834.666666,9303174250.666666,86.033816758,96.44024266973334,0.162363268,0.18646668090666665,
crankshaft,0.55215919,672128.0,66560698064.0,60872611472.0,226.324735659,247.0382051838,0.702757981,1.10772326917,0.55215919,672128.0,27557027216.0,21867531968.0,301.815851975,319.5737722427334,0.412457989,0.571703161062,
frame,0.20678412,451664.0,35846098864.0,30530365936.0,51.605291677,55.64030601819999,0.260485433,0.2938711942126666,0.20678412,451664.0,14737804677.333334,9421984389.333334,83.682453192,89.68407558860001,0.152988838,0.16440537264066665,
falcon,0.37663098,609104.0,59486730112.0,49287090592.0,91.828674716,97.38623786300002,0.447416435,0.524703641048,0.37663098,609104.0,24735836373.333332,14534790645.333334,122.491540174,139.05265751213332,0.227795579,0.24757387427266664,
lathe part,0.48186576,535054.0,57066237056.0,50334709376.0,218.327293967,225.53199235626667,0.576327016,0.6536777842326668,0.48186576,535054.0,26607880037.333332,19876352357.333332,265.505749271,293.9489174844667,0.368252351,0.5049920256866667,
\end{filecontents*}
\csvreader[
        head to column names,
        tabular = >{\em}lccccccccc,
        table head = \toprule& &&\multicolumn{3}{c}{$\mathcal{H}$-matrix} & \multicolumn{4}{c}{Improvement factor $\mathcal{H}/\mathcal{UH}$} \\ \cmidrule(r){4-6}     \cmidrule(r){7-10}\emph{Shape} & {$N$} & {$\kappa h$} & {$\mathrm{Mem_{tot}}$} & {Build} & {MV} & {$\mathrm{Mem_{tot}}$} & {$\mathrm{Mem_{adm}}$} & {Build} & {MV} \\ \midrule,
        table foot = \bottomrule,
    ]{M166520-data/tab2/compression_table.csv}{}{%
        \shape & 
        \tablenum[round-precision=0, round-mode=places, table-format=7.0]{\fpeval{3*\NtH}} & 
        \tablenum[round-precision=3, round-mode=places, table-format=1.3]{\khH} & 
        \tablenum[round-precision=1, round-mode=places, table-format=2.1]{\fpeval{\memtotH/1e9}} & 
        \tablenum[round-precision=0, round-mode=places, table-format=3.0]{\meanBuildH} & 
        \tablenum[round-precision=3, round-mode=places, table-format=1.3]{\meanMatvecH} & 
        \tablenum[round-precision=2, round-mode=places, table-format=1.2]{\fpeval{\memtotH/\memtotU}} & 
        \tablenum[round-precision=2, round-mode=places, table-format=1.2]{\fpeval{\memadmH/\memadmU}} & 
        \tablenum[round-precision=3, round-mode=places, table-format=1.3]{\fpeval{\meanBuildH/\meanBuildU}} & 
        \tablenum[round-precision=2, round-mode=places, table-format=1.2]{\fpeval{\meanMatvecH/\meanMatvecU}} 
    }
\end{table}
For all shapes, the $\UHm$-matrix format is able to compress the admissible memory with a factor between $2.5$ and $4$. As a result, the total memory usage sees a reduction around $2.4$. This is similar to the factors that can be inferred from the previous experiments. 

The additional construction cost is manageable. In the worst case (of the frame), uniform compression increases the total construction time by 61\%. The $\UHm$-matrix construction time may be reduced by opting for e.g.\ column-pivoted QR with early-exit as discussed in \cref{rmk:rr-decomp}.

The memory reduction positively influences the performance of the matrix-vector products. For both formats, these are done in parallel (see \cref{alg:hmatvec,alg:uhmatvec}). While these factors are lower than the actual compression, they still indicate that the parallelizability of the $\UHm$-matrix-vector product is comparable -- at least in a simple parallel implementation -- to that of the regular $\Hm$-matrix.

\section{Conclusion and future work}\label{sec:7_conclusion}
We have shown that $\Hm$-matrices can be cheaply compressed to the uniform $\Hm$-matrix format in an algebraic manner. This more than halves the total memory usage in the process. By incorporating the initial $\Hm$-matrix construction into the compression, a $\UHm$-matrix can be constructed directly without the need to store the intermediate $\Hm$-matrix. Using this approach, the construction time of the $\UHm$-matrix is comparable to that of the $\Hm$-matrix. The more compact $\UHm$-matrix also results in a faster matrix-vector product.

Future work include experiments on a wider variety of matrices, among which different BEM matrices and a direct comparison with the $\Hm^2$-matrix compression. Additionally, we plan to apply the ideas of the uniform compression to the compact representation for wavenumber-dependent BEM matrices in \cite{Dirckx2022}.

\bibliographystyle{siamplain}
\bibliography{M166520-ref}

\end{document}